\documentclass[11pt,reqno]{amsart}
\usepackage{}
\usepackage{mathrsfs}
\usepackage{amsfonts}
\usepackage{a4wide}
\allowdisplaybreaks \numberwithin{equation}{section}
\usepackage{color}

\newtheorem{theorem}{Theorem}[section]

\newtheorem{lemma}[theorem]{Lemma}

\newtheorem{thm}{Theorem}[section]

\newtheorem{lem}[thm]{Lemma}
\newtheorem{rem}[thm]{Remark}

\newtheorem{prop}[thm]{Proposition}

\theoremstyle{definition}

\newcommand{\R}{\mathbb{R}}
\newcommand{\ds}{\displaystyle}

\begin{document}
\title[Critical elliptic  systems ]{construction of solutions to a nonlinear critical elliptic system via local pohozaev identities}

\author{ Qidong Guo, Qingfang Wang$^{\dagger}$ and Wenju Wu}
\address[Qidong Guo]{School of Mathematics and Statistics, Central China Normal
University, Wuhan, 430079, P. R. China}
\email{qdguo@mails.ccnu.edu.cn}
\address[Qingfang Wang]{School of Mathematics and Computer Science, Wuhan Polytechnic University, Wuhan, 430079, P. R. China}
\email{hbwangqingfang@163.com}
\address[Wenju Wu]{School of Mathematics and Statistics, Central China Normal
University, Wuhan, 430079, P. R. China}
\email{wjwu@mails.ccnu.edu.cn}
\thanks{$^{\dagger}$ Corresponding author: Qingfang Wang}
\begin{abstract}
In this paper, we investigate the following elliptic system with Sobolev critical growth
\begin{equation*}
\begin{cases}
-\Delta u+P(|y'|,y'')u=u^{2^*-1}+\frac{\beta}{2} u^{\frac{2^*}{2}-1}v^{\frac{2^*}{2}},\  y\in \R^N,\vspace{0.12cm}\\
-\Delta v+Q(|y'|,y'')v=v^{2^*-1}+\frac{\beta}{2} v^{\frac{2^*}{2}-1}u^{\frac{2^*}{2}}, \  y\in \R^N,\vspace{0.12cm}\\
u,\,v>0, \, u,\,\,v\in H^1(\R^N),
\end{cases}
\end{equation*}
where~$(y',y'')\in \R^2\times\R^{N-2}$, $P(|y'|,y''), Q(|y'|,y'')$ are bounded non-negative function in $\R^+\times\R^{N-2}$, $2^*=\frac{2N}{N-2}$. By combining a finite reduction argument and local Pohozaev type of identities, assuming that $N\geq 5$ and $r^2(P(r,y'')+\kappa ^2Q(r,y''))$ have a common topologically nontrivial critical point,
we construct an unbounded sequence of non-radial positive vector solutions of synchronized type, whose energy can be made arbitrarily large. Our result extends the result of a single critical problem by
[Peng, Wang and Yan,J. Funct. Anal. 274: 2606-2633, 2018]. The novelties mainly include the following two aspects.
On one hand, when $N\geq5$, the coupling exponent $\frac{2}{N-2}<1$, which creates a great trouble for us to apply the perturbation argument directly. This constitutes the main difficulty different between the coupling system and a single equation.
On the other hand, the weaker symmetry conditions of $P(y)$ and $Q(y)$ make us not estimate directly the corresponding derivatives of the reduced
functional in locating the concentration points of the solutions, we employ some local Pohozaev
identities to locate them.


\end{abstract}
\maketitle

\section{Introduction and main results}
\setcounter{equation}{0}
In this paper, we are concerned with the multiplicity of solutions for the following nonlinear critical Schr\"odinger system
\begin{equation}\label{eqs1.1}
\begin{cases}
-\Delta u+P(|y'|,y'')u=u^{2^*-1}+\frac{\beta}{2} u^{\frac{2^*}{2}-1}v^{\frac{2^*}{2}},\  y\in \R^N,\vspace{0.12cm}\\
-\Delta v+Q(|y'|,y'')v=v^{2^*-1}+\frac{\beta}{2} v^{\frac{2^*}{2}-1}u^{\frac{2^*}{2}}, \  y\in \R^N,\vspace{0.12cm}\\
u,\,v>0, \, u,\,\,v\in H^1(\R^N),
\end{cases}
\end{equation}
where $P(|y'|,y''), Q(|y'|,y'')$ are bounded non-negative function in $\R^+\times\R^{N-2}$, $2^*=\frac{2N}{N-2}$, $\beta$ is a coupling constant.
For the understanding of \eqref{eqs1.1}, let us begin with the single equation  in $\R^N$. It is related to the following well-known Brezis-Nirenberg problem in $S^N$
\begin{align}\label{eqs1.2}
-\Delta_{S^N}u=u^{\frac{N+2}{N-2}}+\lambda u,u>0\,\,\,\,\hbox{on}\, S^N.
\end{align}
Indeed, after using the stereographic projection, problem \eqref{eqs1.2} can be reduced to
\begin{align*}
-\Delta u+V(y)u=u^{\frac{N+2}{N-2}},\,\,u>0,\,\,u\in H^1(\R^N)
\end{align*}
with $V(y)=\frac{-4\lambda-N(N-2)}{(1+|y|^2)^2}$ and $V(y)>0$ if $\lambda<-\frac{N(N-2)}{4}$.
In the pioneer work, Chen, Wei and Yan \cite{Chen-Wei-Yan} applied the reduction argument to consider the Schr\"odinger equation with the critical exponent
\begin{align*}
-\Delta u+V(|x|)u=u^{\frac{N+2}{N-2}},\,\,\,\hbox{in}\,\,\,\R^N.
\end{align*}
If $V$ is radially symmetric, $r^2V(r)$ has a local maximum point or a local minimum point $r_0>0$ with $V(r_0)>0$, they proved the existence of infinitely many positive solutions.
Do the same thing for the following equation
\begin{align*}
-\Delta u+V(|y'|,y'')u=u^{\frac{N+2}{N-2}},\,\,\,\,\,\,u\in H^1(\R^N),
\end{align*}
where $(y',y'')\in \R^2\times\R^{N-2}$, $V(|y'|,y'')>0$ is bounded. By using the Pohozaev identities, Peng, Wang and Yan \cite{Peng-Wang-Yan} overcome the difficulty appearing in using the standard reduction method to find algebraic equations which determine the location of bubbles. Moreover, He, Wang and Wang combining the reduction method and local Pohozaev identities proved non-degeneracy of multi-bubbling solutions of \cite{He-Wang-Wang}. On the other hand, F. Du etc. \cite{Du-Hua-Wang-Wang} consider the multi-piece of bubble solutions for a nonlinear critical elliptic equation.

For the system, Peng and Wang discuss the following system in \cite{Peng-Wang}
\begin{equation}\label{eqs1.3}
\begin{cases}
-\Delta u+P(|x|)u=\mu u^3+\beta uv^2,\,\,\,\hbox{in}\,\,\,\R^3,\vspace{0.12cm}\cr
-\Delta v+Q(|x|)v=\nu v^3+\beta u^2v,\,\,\,\hbox{in}\,\,\,\R^3,
\end{cases}
\end{equation}
where $P(r),\,Q(r)$ are positive radial potential, $\mu>0,\,\nu>0$. They examined the effect of nonlinear coupling on the solution structure. The authors use a finite dimensional reduction argument and found the location of bubbles by using a maximization procedure. They constructed an unbounded sequence of non-radial positive vector solutions of segregated type and an unbounded sequence of non-radial positive vector solution of synchronized type. The work of \cite{Lin-Wei} gives solutions with one component peaking at the origin and the other having a finite number of peaks on a k-polygon. On the other hand, for the $\beta>0$, $P=Q=1$, it is known (for example \cite{Bart-Wang,Bart-Dan-Wang}) that there are special positive solutions with the two components being positive constant multiples of the unique positive solution of the scalar cubic Schr\"odinger equation $-\Delta w+w=w^3,\,\,x\in \R^3$. For the elliptic system of subcritical exponent, readers can refer to \cite{Li-Wei-Wu,Liu-Wang-Wang,Pis-Vai,Wang-Ye, Wei-Wu}.

Systems of nonlinear Schr\"odinger equations with critical exponent have been the subject of extensive mathematical studies in recent years, for example,
\cite{a-c,d-w-w,l-w,l-w2,liu-w,sirakov,t-v,t-w,w-y}.
In \cite{g-l-w}, Guo, Li and Wei proved the existence of infinitely many positive non-radial solutions of the coupled elliptic system.
Also, Chen, Medina and Pistoia \cite{cmp-2022} consider the segregated solutions for a critical elliptic system with a small interspecies repulsive force in $\R^4$.
More recently, Del Pino, Musso Pacard and Pistoia \cite{dmpp1,dmpp2} constructed sequences of sign-changing solutions with large energy and concentrating along some special sub-manifold of $\R^N$.
However, so far, the existing results of using the Lyapunov-Schmidt reduction argument to construct  concentrated solutions of an elliptic system have been basically about low dimensional cases with $N<5$, such as \cite{cpv-2022,cmp-2022,g-l-w,Peng-Wang,Peng-Wang-Wang,wxzz-2015-amss}, etc.
Especially in \cite{Peng-Peng-Wang}, Peng, Peng and Wang consider the Dirichlet problem with Sobolev critical exponent
\begin{equation}\label{eqs1.3}
\begin{cases}
-\Delta u=|u|^{2^*-2}u+\frac{\alpha}{2^*}|u|^{\alpha-2}u|v|^{\beta},\,\,\,\,\hbox{in}\,\,\,\Omega,\vspace{0.12cm}\cr
-\Delta v=|v|^{2^*-2}v+\frac{\beta}{2^*}|u|^\alpha |v|^{\beta-2}v,\,\,\,\,\hbox{in}\,\,\,\Omega,
\end{cases}
\end{equation}
where $\alpha+\beta=2^*=\frac{2N}{N-2}$ and $\Omega=\R^N$ or $\Omega$ is a smooth bounded domain in $\R^N$. They obtain a uniqueness result on the least energy solutions and show that a manifold of the synchronized type of positive solutions is non-degenerate for the above system for some ranges of the parameters $\alpha,\beta,N$.

Motivated by \cite{Guo-Wang-Wang,Peng-Wang-Yan}, in this paper, we mainly want to study the existence of infinitely many solutions of system \eqref{eqs1.1} with weaker symmetry of potentials. We would like to point out that very recently in \cite{Guo-Wang-Wang} using the finite dimensional reduction method obtained infinitely many non-radial solutions for system \eqref{eqs1.1} with radial potentials. Also, about the single critical Schrodinger equation $-\Delta u+V(y)u=u^{\frac{N+2}{N-2}},y\in\R^N$ involving with a potential linear term, there are some existence of solutions for it, one can refer \cite{Chen-Wei-Yan, Du-Hua-Wang-Wang, Wang-Wang-Yang}, in particular, in \cite{Peng-Wang-Yan} Peng, Wang and Yan first applied the finite dimensional reduction combining the local Pohozaev identities obtained infinitely many solutions of \eqref{eqs1.1}.

In this paper, we assume that $P(y),\,Q(y)$ satisfy the following assumptions:
\begin{itemize}
\item[$(i)$] The function $r^2(P(r,y'')+\kappa^2Q(r,y''))$ have a common critical point $(r_0,y_0'')$ such that $r_0>0$ and $P(r_0,y_0'')>0$ and $Q(r_0,y_0'')>0$;\\
\item[$(ii)$] $deg\big(\nabla(r^2P(r,y'')+\kappa^2r^2Q(r,y'')),(r_0,y_0'')\big)\neq 0$;
\item[$(iii)$] $\beta$ satisfies $1+\beta\kappa^{\frac{2^*}{2}}>0,$
\end{itemize}
where $\kappa$ satisfies
\begin{equation}\label{kappa}
2+\kappa^{\frac{2^*}{2}}-\beta\kappa^{\frac{2^*}{2}-2}-2\kappa^{2^*-2}=0.
\end{equation}

Our main result in this paper can be stated as follows.
\begin{thm}\label{th1}
Suppose $N\geq5$ and $P\geq0,\,Q\geq0$ are bounded and belong to $C^1$. If $P(r,y''),\,Q(r,y'')$ satisfies $(i)$ and $(ii)$,  then problem \eqref{eqs1.1} has infinitely many non-radial positive solutions whose energy can be made arbitrarily large.
\end{thm}

\medskip
\begin{rem}
It is worth noting that in the higher dimensional space with $N\geq5$, the exponent $\frac2{N-2}<1$ from the worse couple terms will create a great trouble for us to apply the perturbation argument  directly.
To overcome this difficulty, we should ensure that the error term can be controlled by a small multiple of the approximate solution (see Lemma \ref{lmA.6}).
\end{rem}

\begin{rem}
We would like to point out that the weaker symmetry conditions of $P(y)$ and $Q(y)$ make us not estimate directly the corresponding derivatives of the reduced
functional in locating the concentration points of the solutions, we employ some local Pohozaev
identities to locate them.
Moreover,  compared with \cite{Guo-Wang-Wang}, the system we consider here involving a linear term which causes the problem more difficult.
\end{rem}

\begin{rem}
The non-radial vector solutions of separated type for system \eqref{eqs1.1}
is also very important but more challenging and we will consider this topic in a future work.
\end{rem}

\medskip
It is well known that the only solutions of the problem
\begin{align*}
-\Delta w=w^{\frac{N+2}{N-2}},\,\,\,w>0\,\,\,\hbox{in}\,\,\, \R^N
\end{align*}
are of the form
\begin{align*}
w_{x,\lambda}(y):=\lambda^{\frac{N-2}2} w(\lambda(y-x))=(N(N-2))^{\frac{N-2}4}\Bigl(\frac\lambda{1+\lambda^2|y-x|^2}\Bigr)^{\frac{N-2}2},\ \ \lambda>0,\ x\in\R^N.
\end{align*}
Moreover, by \cite{Peng-Peng-Wang}, then $N+1$ dimensional manifold
\begin{align*}
\{(U_{x_0,\lambda},V_{x_0,\lambda})=(s w_{x_0,\lambda}, tw_{x_0,\lambda}), t=\kappa s, x_0\in\R^N, \lambda>0\}
\end{align*}
is a non-degenerate solution of the following equations
\begin{equation}\label{weqs1.2}
\begin{cases}
-\Delta u=u^{2^*-1}+\frac{\beta}{2} u^{\frac{2^*}{2}-1}v^{\frac{2^*}{2}},\  y\in \R^N,\vspace{0.12cm}\\
-\Delta v=v^{2^*-1}+\frac{\beta}{2} v^{\frac{2^*}{2}-1}u^{\frac{2^*}{2}}, \  y\in \R^N,\vspace{0.12cm}\\
u,\,v>0, \, u,\,\,v\in D^{1,2}(\R^N),
\end{cases}
\end{equation}
where $s$ satisfies that $s^{2^*-2}=\frac{2}{2+\beta\kappa^{\frac{2^*}{2}}}$ and $\kappa$ satisfies \eqref{kappa}.

Let $y=(y',y''), y'\in\R^2, y''\in\R^{N-2}$. We define
\begin{align*}
H_s=\Bigl\{u:u\in D^{1,2}(\R^N), &u\ \hbox{is\ even\ in} \ y_h,h=2,\cdots,N,\\
& u(r\cos\theta,rsin\theta,y'')=u(r\cos(\theta+\frac{2\pi j}k),r\sin(\theta+\frac{2\pi j}k),y'')
\Bigr\}.
\end{align*}
For any large integer $k>0$, let
$$x_j=\Bigl(r\cos\frac{2(j-1)\pi}{k},r\sin\frac{2(j-1)\pi}{k},\mathbf{0}\Bigr),\ \ j=1,\ldots,k,$$
where $\mathbf{0}$ is the zero vector in $\R^{N-2}$.

Let $\delta>0$ is a small constant, such that
\begin{align}\label{weqs1.3}
r^2P(r,y'')>0,\,\,\,r^2Q(r,y'')>0,\,\,\,\hbox{if}\,\,\,|(r,y'')-(r_0,y_0'')|\leq 10\delta.
\end{align}

To deal with the slow decay of this
function when $N$ is not big, we need to cut off this function. Let $\xi(y)=\xi(|y'|,y'')$ be a smooth function satisfying $\xi=1$ if $|(r,y'')-(r_0,y_0'')|\leq\delta$; $\xi=0$ if $|(r,y'')-(r_0,y_0'')|\geq 2\delta$ and $0\leq\xi\leq1$.
Denote
\begin{align*}
W_{1,x_j}=\xi U_{x_j,\lambda},\,\,\,W_{1,\lambda}^*=\sum\limits_{j=1}^kU_{x_j,\lambda},\,\,\,W_{1,\lambda}=\sum\limits_{j=1}^kW_{1,x_j},
\end{align*}
\begin{align*}
W_{2,x_j}=\xi V_{x_j,\lambda},\,\,\,W_{2,\lambda}^*=\sum\limits_{j=1}^kV_{x_j,\lambda},\,\,\,W_{2,\lambda}=\sum\limits_{j=1}^kW_{2,x_j}.
\end{align*}

In this paper, we always assume that $k>0$ is a large integer, $\lambda\in \Bigl[L_0k^{\frac{N-2}{N-4}},L_1k^{\frac{N-2}{N-4}}\Bigr]$
for some constant $L_1>L_0>0$.
We will prove Theorem \ref{th1} by verifying the following result.
\begin{thm}\label{thm-add}
Under the assumptions of Theorem \ref{th1}, there exists a positive integer $k_0>0$ such that for any integer $k\geq k_0$, \eqref{eqs1.1} has a solution $(u_k,v_k)$ of the form
\begin{align*}
u_k=W_{1,\lambda_k}+\varphi_k=\sum\limits_{j=1}^k\xi U_{x_j,\lambda_k}+\varphi_k,\,\,\,
v_k=W_{2,\lambda_k}+\psi_k=\sum\limits_{j=1}^k\xi V_{x_j,\lambda_k}+\psi_k,
\end{align*}
where $(\varphi_k,\psi_k)\in H_s\times H_s$. Moreover, as $k\rightarrow+\infty$, $\lambda_k\in \Bigl[L_0k^{\frac{N-2}{N-4}},L_1k^{\frac{N-2}{N-4}}\Bigr]$, $(\bar{r}_k,\bar{y}''_k)\rightarrow(r_0,y_0'')$ and $\lambda_k^{-\frac{N-2}{2}}\|(\varphi_k,\psi_k)\|_{L^\infty}\to 0$.
\end{thm}

Before we close this section, let us explain the main ideas in the proof of Theorem \ref{thm-add}. We will use a finite reduction argument in the proof.
To deal with the large number of bubbles in the
solution, similar as \cite{Wei-Yan-10-CV,Wei-Yan},  the reduction procedure is carried out in a weighted space instead of the standard Sobolev space.
Since the critical points including the saddle points of the potentials, which makes us not estimate directly the corresponding derivatives of the reduced
functional in locating the concentration points of the solutions, motivated by \cite{Peng-Wang-Yan} we employ some local Pohozaev
identities to locate them.

This paper is organized as follows. In section \ref{s2}, we carry out the  reduction procedure with weighted maximum norm to a finite dimensional setting.
 In section \ref{s3}, we solve the corresponding finite dimensional probem and prove the main result Theorem \ref{th1}. Some delicate estimates and tools are put in the appendix.

\section{Finite-dimensional reduction}\label{s2}
In this section, we perform a finite-dimensional reduction. For simplicity, we write $2^*=\frac{2N}{N-2}$.
Set
\begin{align*}
\|u\|_*= \sup_{y\in\R^N}\Big(\sum_{j=1}^k\frac{1}{(1+\lambda|y-x_j|)^{\frac{N-2}{2}+\tau}}\Big)^{-1}\lambda^{-\frac{N-2}{2}}|u(y)|
\end{align*}
and
\begin{align*}
\|f\|_{**}= \sup_{y\in\R^N}\Big(\sum_{j=1}^k\frac{1}{(1+\lambda|y-x_j|)^{\frac{N+2}{2}+\tau}}\Big)^{-1}\lambda^{-\frac{N+2}{2}}|f(y)|,
\end{align*}
where $\tau=\frac{N-4}{N-2}$.

We also denote that $\|(u,v)\|_*=\|u\|_*+\|v\|_*$
and
$\|(f,g)\|_{**}=\|f\|_{**}+\|g\|_{**}.$

Let
\begin{align*}
Y_{j,1}=\frac{\partial W_{1,x_j}}{\partial r},\,\,\,Y_{j,2}=\frac{\partial W_{1,x_j}}{\partial \lambda},\,\,\,Y_{j,k}=\frac{\partial W_{1,x_j}}{\partial \bar{y}''_k},\,\,\,k=3,\cdots,N,\cr
Z_{j,1}=\frac{\partial W_{2,x_j}}{\partial r},\,\,\,Z_{j,2}=\frac{\partial W_{2,x_j}}{\partial \lambda},\,\,\,Z_{j,k}=\frac{\partial W_{2,x_j}}{\partial \bar{y}''_k},\,\,\,k=3,\cdots,N.
\end{align*}
First we consider the linear problem
\begin{align}\label{eqlinear}
\begin{cases}
L_k(\varphi,\psi)=(h,g)+\sum\limits_{i=1}^Nc_i\sum\limits_{j=1}^k(W_{1,x_j}^{2^*-2}Y_{j,i},W_{2,x_j}^{2^*-2}Z_{j,i}),\\
(\varphi,\psi)\in H_s\times H_s,\vspace{0.12cm}\\
\Bigl\langle (W_{1,x_j}^{2^*-2}Y_{j,l},W_{2,x_j}^{2^*-2}Z_{j,l}),(\varphi,\psi)\Bigr\rangle=0,\ \ j=1,\cdots,k,\ l=1,2,\cdots,N
\end{cases}
\end{align}
for some number $c_i$, where $\langle(u_1,u_2),(v_1,v_2)\rangle= \ds\int_{\R^N}(u_1v_1+u_2v_2)$,
\begin{small}
\begin{align}\label{L}
&L_k(\varphi,\psi)=\\
&\left(\begin{matrix}
-\Delta\varphi+P(r,y'')\varphi-(2^*-1)W_{1,\lambda}^{\frac{4}{N-2}}\varphi
-\frac{\beta}{(N-2)} W_{1,\lambda}^{\frac{4-N}{N-2}}W_{2,\lambda}^{\frac{N}{N-2}}\varphi-\frac {N\beta}{2(N-2)} W_{1,\lambda}^{\frac{2}{N-2}}W_{2,\lambda}^{\frac{2}{N-2}}\psi\nonumber \\
-\Delta\psi+Q(r,y'')\psi-(2^*-1)W_{2,\lambda}^{\frac{4}{N-2}}\psi
-\frac{\beta}{(N-2)} W_{1,\lambda}^{\frac{4-N}{N-2}}W_{2,\lambda}^{\frac{N}{N-2}}\psi-\frac {N\beta}{2(N-2)} W_{1,\lambda}^{\frac{2}{N-2}}W_{2,\lambda}^{\frac{2}{N-2}}\varphi
\end{matrix}\right)^\top\nonumber \\=
&
(\varphi,\psi)\mathcal L_k^\top,
\end{align}
\end{small}
with
\begin{small}
\begin{align*}&\mathcal L_k=\\
&\left(\begin{matrix}
-\Delta+P(r,y'')-(2^*-1)W_{1,\lambda}^{\frac{4}{N-2}}
-\frac\beta{N-2} W_{1,\lambda}^{\frac{4-N}{N-2}}W_{2,\lambda}^{\frac{N}{N-2}} &-\frac{N\beta}{2(N-2)} W_{1,\lambda}^{\frac{2}{N-2}}W_{2,\lambda}^{\frac{2}{N-2}}\nonumber \\
-\frac{N\beta}{2(N-2)} W_{1,\lambda}^{\frac{2}{N-2}}W_{2,\lambda}^{\frac{2}{N-2}} &
-\Delta+Q(r,y'')-(2^*-1)W_{2,\lambda}^{\frac{4}{N-2}}
-\frac{\beta}{N-2} W_{2,\lambda}^{\frac{4-N}{N-2}}W_{1,\lambda}^{\frac{N}{N-2}}
\end{matrix}\right).
\end{align*}
\end{small}

\begin{lem}\label{lemlinear}
Assume that $(\varphi,\psi)$ solves \eqref{eqlinear}. If $\|(h_1,h_2)\|_{**}$ goes to zero as $k$ goes to infinity, so does
$\|(\varphi,\psi)\|_{*}$.
\end{lem}

\begin{proof}
By contradiction, we assume that there exist $k\rightarrow\infty$, $(\varphi_k,\psi_k)$,  $\lambda_k\in \Bigl[L_0k^{\frac{N-2}{N-4}},L_1k^{\frac{N-2}{N-4}}\Bigr]$, some $(\varphi_k,\psi_k)$ solving \eqref{eqlinear}, with  $\|(h_{1,k},h_{2,k})\|_{**}\rightarrow0$ and
$\|(\varphi_k,\psi_k)\|_{*}\geq c'>0$. We may assume that $\|(\varphi_k,\psi_k)\|_{*}=1$.
For simplicity, we drop the subscript $k$ and write \eqref{eqlinear} as
\begin{align}\label{eqinte1}
\varphi(y)&=\int_{\R^N}\frac{(2^*-1)}{|y-z|^{N-2}}W_{1,\lambda}^{\frac{4}{N-2}}\varphi(z)dz\nonumber \\
&\quad+\int_{\R^N}\frac{1}{|y-z|^{N-2}}\Big(\frac 1{(N-2)} W_{1,\lambda}^{\frac{4-N}{N-2}}W_{2,\lambda}^{\frac{N}{N-2}}\varphi(z)+\frac N{2(N-2)} W_{1,\lambda}^{\frac{2}{N-2}}W_{2,\lambda}^{\frac{2}{N-2}}\psi(z)\Big)dz\nonumber  \\
&\quad+\int_{\R^N}\frac{1}{|y-z|^{N-2}}\Big(h_1(z)+\sum_{i=1}^2c_i\sum_{j=1}^k w_{x_j,\lambda}^{2^*-2}Y_{j,i}\Big)dz
\end{align}
and
\begin{align}\label{eqinte2}
\psi(y)&=\int_{\R^N}\frac{(2^*-1)}{|y-z|^{N-2}}W_{2,\lambda}^{\frac{4}{N-2}}\psi(z)dz\nonumber \\
&\quad+\int_{\R^N}\frac{1}{|y-z|^{N-2}}\Big(\frac 1{(N-2)} W_{2,\lambda}^{\frac{4-N}{N-2}}W_{1,\lambda}^{\frac{N}{N-2}}\psi(z)+\frac N{2(N-2)} W_{2,\lambda}^{\frac{2}{N-2}}W_{1,\lambda}^{\frac{2}{N-2}}\varphi(z)\Big)dz\nonumber \\
&\quad+\int_{\R^N}\frac{1}{|y-z|^{N-2}}\Big(h_2(z)+\sum_{i=1}^2c_i\sum_{j=1}^k w_{x_j,\lambda}^{2^*-2}Z_{j,i}\Big)dz.
\end{align}

We are sufficed to deal with \eqref{eqinte1}, since for \eqref{eqinte2}, one can obtain similar estimates just in the same way.
From Lemma \ref{lmA.5}, we first get that
\begin{align}\label{1}
&\Big|\int_{\R^N}{|y-z|^{N-2}}W_{1,\lambda}^{\frac{4}{N-2}}\varphi(z)dz\Big|\nonumber \\
 \leq &C\|\varphi\|_*\int_{\R^N}\frac{1}{|y-z|^{N-2}}W_{1,\lambda}^{\frac{4}{N-2}}\sum_{j=1}^k\frac{\lambda^{\frac{N-2}{2}}}{(1+\lambda|z-x_j|)^{\frac{N-2}{2}+\tau}}dz\cr
 \leq &C\|\varphi\|_*\lambda^{\frac{N-2}{2}}\sum_{j=1}^k\frac{1}{(1+\lambda|y-x_j|)^{\frac{N-2}{2}+\theta}}.
\end{align}
Recalling that $W_{2,\lambda}=\kappa W_{1,\lambda}$, we obtain also
\begin{align}\label{2}
&\Big|\int_{\R^N}\frac{1}{|y-z|^{N-2}}\Big(\frac 1{(N-2)} W_{1,\lambda}^{\frac{4-N}{N-2}}W_{2,\lambda}^{\frac{N}{N-2}}\varphi(z)+\frac N{2(N-2)} W_{1,\lambda}^{\frac{2}{N-2}}W_{2,\lambda}^{\frac{2}{N-2}}\psi(z)\Big)dz\Big|\nonumber  \\
&\leq C\lambda^{\frac{N-2}{2}}\|(\varphi,\psi)\|_*\sum_{j=1}^k\frac{1}{(1+\lambda|y-x_j|)^{\frac{N-2}{2}+\theta+\tau}}.
\end{align}
Moreover, from Lemma \ref{lmA.4}, there hold
\begin{align}\label{3}
\Big|\int_{\R^N}\frac{1}{|y-z|^{N-2}}h_1(z)dz\Big|
\leq C\lambda^{\frac{N-2}{2}}\|h_1\|_{**}\sum_{j=1}^k\frac{1}{(1+\lambda|y-x_j|)^{\frac{N-2}{2}+\bar\sigma}}
\end{align}
and
\begin{align}\label{4}
&\Big|\int_{\R^N}\frac{1}{|y-z|^{N-2}}\sum_{j=1}^kW_{1,x_j}^{2^*-2}Y_{j,i}dz\Big|
 \leq C\lambda^{\frac{\lambda-2}{2}+n_i} \sum_{j=1}^k\frac{1}{(1+\lambda|y-x_j|)^{\frac{N-2}{2}+\tau}},
\end{align}
where $n_j=1, j=2,\cdots,N$, $n_1=-1$.

Now, we estimate $c_l,l=1,\cdots,N$. Multiplying the two equations in \eqref{eqlinear} by $Y_{1,l}$ and $Z_{1,l}$ respectively and integrating, we find that
\begin{align}\label{cl}
\sum_{i=1}^N\sum_{j=1}^k\left\langle(W_{1,x_j}^{2^*-2}Y_{j,i},W_{2,x_j}^{2^*-2}Z_{j,i}),(Y_{1,l},Z_{1,l})\right\rangle c_i
=\left\langle L_k(\varphi_1,\varphi_2),(Y_{1,l},Z_{1,l})\right\rangle-\left\langle (h_1,h_2),(Y_{1,l},Z_{1,l})\right\rangle.
\end{align}
From Lemma \ref{lmA.3} and since $\bar\sigma>1$, we have
\begin{align*}
\left|\left\langle (h_1,h_2),(Y_{1,l},Z_{1,l})\right\rangle\right|
&\leq C \|(h_1,h_2)\|_{**}\int_{\R^N}\frac{\lambda^{\frac{N-2}{2}+n_l}}{(1+|z-x_1|)^{N-2}}\sum_{j=1}^k\frac{\lambda^{\frac{N+2}{2}}}{(1+|z-x_j|)^{\frac{N+2}{2}+\tau}}dz\\
&\leq C\lambda^{n_l} \|(h_1,h_2)\|_{**}.
\end{align*}
By direct computation, we have
\begin{align}\label{eqs2.10}
|\langle P(r,y'')\varphi,Y_{1,l}\rangle|\leq C\|\varphi\|_{*}\int_{\R^N}\frac{\xi\lambda^{\frac{N-2}{2}+n_i}}{(1+\lambda|z-x_1|)^{(N-2)}}\sum\limits_{j=1}^k\frac{\lambda^{\frac{N+2}{2}}}{1+\lambda|y-x_j|^{\frac{N+2}{2}+\tau}}
=O\Bigl(\frac{\lambda^{n_i}\|\varphi\|_{*}}{\lambda^{1+\epsilon}}\Bigr).
\end{align}
Similarly, we have
\begin{align*}
\Bigl|\langle Q(r,y'')\psi,Z_{1,l}\rangle\Bigr|\leq O\Bigl(\frac{\lambda^{n_l}\|\psi\|_{*}}{\lambda^{1+\varepsilon}}\Bigr).
\end{align*}
On the other hand, we also have
\begin{align}\label{eqs2.11}
\left\langle L_k(\varphi,\psi),(Y_{1,l},Z_{1,l})\right\rangle
=\left\langle L_k(Y_{1,l},Z_{1,l}),(\varphi,\psi)\right\rangle.
\end{align}
Then
\begin{align*}
\eqref{eqs2.11}=&\frac{1}{(N-2)}\int_{\R^N}\Bigl[(N+2)\Bigl(W_{1,x_1}^{\frac4{N-2}}-W_{1,\lambda}^{\frac4{N-2}}\Bigr)
+\beta\Bigl(W_{1,x_1}^{\frac{4-N}{N-2}}W_{2,x_1}^{\frac N{N-2}}-W_{1,\lambda}^{\frac{4-N}{N-2}}W_{2,\lambda}^{\frac N{N-2}}\Bigr)\cr
&+\frac {N\beta}{2}\Bigl(W_{1,x_1}^{\frac{2}{N-2}}W_{2,x_1}^{\frac 2{N-2}}-W_{1,\lambda}^{\frac{2}{N-2}}W_{2,\lambda}^{\frac 2{N-2}}\Bigr)\Bigr]Y_{1,l}\varphi dy\\
&+\frac1{N-2}\int_{\R^N}\Bigl[(N+2)\Bigl(W_{1,x_1}^{\frac4{N-2}}-W_{2,\lambda}^{\frac4{N-2}}\Bigr)
+\Bigl(W_{2,x_1}^{\frac{4-N}{N-2}}W_{1,x_1}^{\frac N{N-2}}-W_{2,\lambda}^{\frac{4-N}{N-2}}W_{1,\lambda}^{\frac N{N-2}}\Bigr)\\
&+\frac{N\beta}{2}\Bigl(W_{2,x_1}^{\frac{2}{N-2}}W_{1,x_1}^{\frac 2{N-2}}-W_{2,\lambda}^{\frac{2}{N-2}}W_{1,\lambda}^{\frac 2{N-2}}\Bigr)\Bigr]Z_{1,l}\psi dy.
\end{align*}
Then, we have
\begin{align}\label{2-5}
\left\langle L_k(\varphi,\psi),(Y_{1,l},Z_{1,l})\right\rangle=O\Big(\frac{\lambda^{n_l}}{\lambda^{1+\varepsilon}}\Big)\|(\varphi,\psi)\|_*+\|(h_1,h_2)\|_{**}.
\end{align}

But, observe that, there exists some $\bar c>0$ such that
\begin{align*}
\sum_{j=1}^k\big\langle(W_{1,x_j}^{2^*-2}Y_{j,i},W_{2,x_j}^{2^*-2}Z_{j,i}),(Y_{1,l},Z_{1,l})\big\rangle
=\lambda^{2n_l}(\bar c+o(1))\delta_{il}.
\end{align*}
Therefore, from \eqref{cl}, we get that
\begin{align}\label{6}
c_l=O\left(\frac1{\lambda^{n_l}}\|(\varphi,\psi)\|_*+\|(h_1,h_2)\|_{**}\right).
\end{align}

Combining \eqref{1}-\eqref{6}, we obtain that
\begin{align}\label{varphiest}
\|(\varphi,\psi)\|_*\leq\Bigl(\|(h_1,h_2)\|_{**}+\frac{\sum_{j=1}^k
\frac{1}{(1+\lambda|y-x_j|)^{\frac{N-2}{2}+\tau+\theta}}}{\sum_{j=1}^k\frac{1}{(1+\lambda|y-x_j|)^{\frac{N-2}{2}+\tau}}}\Bigr).
\end{align}

Since $\|(\varphi,\psi)\|_*=1$, we get from
\eqref{varphiest} that there exists some $R,a>0$ such that for some $j$
\begin{align}\label{7}
\|(\varphi,\psi)\|_{L^\infty(B_R(x_j))}\geq a>0.
\end{align}
However, transformation $(\bar\varphi(y),\bar\psi(y))=(\varphi(y-x_j),\psi(y-x_j))$, notice that converges uniformly in any compact set to a solution
$(u,v)$ which satisfies
\begin{align}\label{eqlimit}
\begin{cases}
-\Delta u-(2^*-1)U_{0,\lambda}^{\frac{4}{N-2}}u
-\frac1{N-2} U_{0,\lambda}^{\frac{4}{N-2}}V_{0,\lambda}^{\frac{N}{N-2}}u-\frac {N\beta}{2(N-2)} U_{0,\lambda}^{\frac{2}{N-2}}V_{0,\lambda}^{\frac{2}{N-2}}v=0,\,\,\,y\in \R^{N},\\
-\Delta v-(2^*-1)V_{0,\lambda}^{\frac{4}{N-2}}v
-\frac1{N-2} V_{0,\lambda}^{\frac{4}{N-2}}U_{0,\lambda}^{\frac{N}{N-2}}v-\frac {N\beta}{2(N-2)}V_{0,\lambda}^{\frac{2}{N-2}}U_{0,\lambda}^{\frac{2}{N-2}}u=0,\,\,\,y\in\R^{N}
\end{cases}
\end{align}
for some $\lambda\in[L_1,L_2]$. Moreover, $(u,v)$ is perpendicular to the kernel of \eqref{eqlimit}. Hence, $(u,v)=(0,0)$,
which is a contradiction to \eqref{7}.
\end{proof}
As a result of Lemma \ref{lemlinear}, applying the same argument as in the proof of proposition 4.1 in \cite{DFM}. We can prove the following Proposition.
\begin{prop}\label{proplinear}
There exists $k_0>0$ and some constant $C>0$, independent of $k$, such that for all $k\geq k_0$ and all $(h_1,h_2)\in L^\infty(\R^N)\times L^\infty(\R^N)$,
the linear problem \eqref{eqlinear} has a unique solution $(\varphi,\psi)\equiv \mathbb L_k(h_1,h_2)$. Moreover, there hold that
\begin{align}\label{linearest}
\|\mathbb L_k(h_1,h_2)\|_*\leq C\|(h_1,h_2)\|_{**},\ \ \ \ |c_l|\leq \frac{1}{\lambda^{n_l}}\|(h_1,h_2)\|_{**}.
\end{align}
\end{prop}

\medskip
Now we consider the nonlinear problem
\begin{align}\label{eqnonlinear0}
\begin{cases}
-\Delta (W_{1,\lambda}+\varphi)+P(r,y'')(W_{1,\lambda}+\varphi)=(W_{1,\lambda}+\varphi)^{2^*-1}+\frac{\beta}{2} (W_1+\varphi)^{\frac{2^*}{2}-1}(W_{2,\lambda}+\psi)^{\frac{2^*}{2}},\  \vspace{0.12cm}\\
\,\,\,\,\,\,\,\,\,\,\,\,\,\,\,\,\,\,\,\,\,\,\,\,\,\,\,\,\,\,\,\,\,\,\,\,\,\,\,\,\,\,\,\,\,\,\,\,\,\,\,\,\,\,\,\,\,\,\,\,\,\,\,\,\,\,\,\,\,\,\,\,+\sum\limits_{i=1}^2c_i\sum\limits_{j=1}^kW_{1,x_j}^{2^*-2}Y_{j,i},\\
-\Delta (W_{2,\lambda}+\psi)+Q(r,y'')(W_{2,\lambda}+\psi)=(W_{2,\lambda}+\psi)^{2^*-1}+\frac{\beta}{2} (W_{2,\lambda}+\psi)^{\frac{2^*}{2}-1}(W_{1,\lambda}+\varphi)^{\frac{2^*}{2}}, \ \vspace{0.12cm}\\
\,\,\,\,\,\,\,\,\,\,\,\,\,\,\,\,\,\,\,\,\,\,\,\,\,\,\,\,\,\,\,\,\,\,\,\,\,\,\,\,\,\,\,\,\,\,\,\,\,\,\,\,\,\,\,\,\,\,\,\,\,\,\,\,\,\,\,\,\,\,\,\,\,\,\,\,\,\,+\sum\limits_{i=1}^2c_i\sum\limits_{j=1}^kW_{2,x_j}^{2^*-2}Z_{j,i},\\
(\varphi,\psi)\in H_s\times H_s,\vspace{0.12cm}\\
\big\langle (W_{1,x_j}^{2^*-2}Y_{j,l},W_{2,x_j}^{2^*-2}Z_{j,l}),(\varphi,\psi)\big\rangle=0,\ \ j=1,\ldots,k,\ \ l=1,2,\cdots,N.
\end{cases}
\end{align}

In this section, we are aimed to prove that
\begin{prop}\label{propnonlinear}
There exists $k_0>0$ and some constant $C>0$, independent of $k$, such that for all $k\geq k_0$, $\lambda\in [L_0k^{\frac{N-2}{N-4}},L_1k^{\frac{N-2}{N-4}}]$,
with $\bar\theta>0$ is a fixed small constant,
problem \eqref{eqnonlinear0} has a unique solution $(\varphi,\psi)=(\varphi(r,\lambda),\psi(r,\lambda))$ such that
\begin{align}\label{linearest}
\|(\varphi,\psi)\|_*\leq C\left(\frac1\lambda\right)^{1+\epsilon},\,\,\,|c_l|\leq C\left(\frac1\lambda\right)^{1+\epsilon},\,\,\,
|\varphi|\leq\frac12 U,\,\,\,|\psi|\leq\frac12 V,
\end{align}
where $\epsilon>0$ small enough.
\end{prop}

Rewrite the nonlinear problem \eqref{eqnonlinear0} as
\begin{align}\label{eqnonlinear}
\begin{cases}
L_k(\varphi,\psi)=R_k+N_k(\varphi,\psi),\vspace{0.12cm}\\
(\varphi,\psi)\in H_s\times H_s,\vspace{0.12cm}\\
\big\langle (W_{1,x_j}^{2^*-2}Y_{j,l},W_{2,x_j}^{2^*-2}Z_{j,l}),(\varphi,\psi)\big\rangle=0,\ \ j=1,\ldots,k,\ \ l=1,\cdots,N,
\end{cases}
\end{align}
where the operator $L_k$ is defined in \eqref{L}, and
\begin{align}\label{N}
&N_k(\varphi,\psi)=\left(
N_k^{1,1}+2N_k^{1,2},
N_k^{2,1}+2N_k^{2,2}
\right), \,\,\,R_k=(R_{1,k},R_{2,k}),
\end{align}
with
\begin{small}
\begin{align*}
&N_k^{1,1}=(W_{1,\lambda}+\varphi)^{2^*-1}-W_{1,\lambda}^{2^*-1}-(2^*-1)W_{1,\lambda}^{2^*-2}\varphi,\cr
&N_K^{1,2}=(W_{1,\lambda}+\varphi)^{\frac{2^*}{2}-1}(W_{2,\lambda}+\psi)^{\frac{2^*}{2}}-W_{1,\lambda}^{\frac{2^*}{2}-1}W_{2,\lambda}^{\frac{2^*}{2}}
-\Bigl(\frac {2^*}{2}-1\Bigr) W_{1,\lambda}^{\frac{2^*}{2}-2}W_{2,\lambda}^{\frac{2^*}{2}}\varphi-\frac {2^*}{2} W_{1,\lambda}^{\frac{2^*}{2}-1}W_{2,\lambda}^{\frac{2^*}{2}-1}\psi,\cr
&N_K^{2,1}=(W_{2,\lambda}+\psi)^{2^*-1}-W_{2,\lambda}^{2^*-1}-(2^*-1)W_{2,\lambda}^{2^*-2}\psi,\cr
&N_K^{2,2}=(W_{2,\lambda}+\psi)^{\frac{2^*}{2}-1}(W_{1,\lambda}+\varphi)^{\frac{2^*}{2}}-W_{2,\lambda}^{\frac{2^*}{2}-1}W_{1,\lambda}^{\frac{2^*}{2}}
-\Bigl(\frac{2^*}{2}-1\Bigr) W_{2,\lambda}^{\frac{2^*}{2}-2}W_{1,\lambda}^{\frac{2^*}{2}}\psi-\frac{2^*}{2}W_{2,\lambda}^{\frac{2^*}{2}-1}W_{1,\lambda}^{\frac{2^*}{2}-1}\varphi.
\end{align*}
\end{small}

Denote
\begin{align*}
R_{1,k}=&\Bigl(W_{1,\lambda}^{2^*-1}-\sum\limits_{j=1}^kW_{1,x_j}^{2^*-1}\Bigr)-P(r,y'')W_{1,\lambda}+W_{1,\lambda}^*\Delta\xi+2\nabla\xi\nabla W_{1,\lambda}^*\cr
&+\frac{\beta}{2}\Bigl(W_{2,\lambda}^{\frac{2^*}{2}-1}W_{1,\lambda}^{\frac{2^*}{2}}
-\sum\limits_{j=1}^kW_{2,x_j}^{\frac{2^*}{2}}W_{1,x_j}^{\frac{2^*}{2}}\Bigr)
\end{align*}
and
\begin{align*}
R_{2,k}=&\Bigl(W_{2,\lambda}^{2^*-1}-\sum\limits_{j=1}^kW_{2,x_j}^{2^*-1}\Bigr)-Q(r,y'')W_{2,\lambda}+W_{2,\lambda}^*\Delta\xi+2\nabla\xi\nabla W_{2,\lambda}^*\cr
&+\frac{\beta}{2}\Bigl(W_{1,\lambda}^{\frac{2^*}{2}-1}W_{2,\lambda}^{\frac{2^*}{2}}
-\sum\limits_{j=1}^kW_{2,x_j}^{\frac{2^*}{2}}W_{1,x_j}^{\frac{2^*}{2}}\Bigr).
\end{align*}

In the following, we will use the contraction mapping theorem to show that there exists a unique solution to problem \eqref{eqnonlinear}
in the set that $\|(\varphi,\psi)\|_*$ is small. In order to do this, we first estimate $N_k(\varphi,\psi)$ and $R_k$,
and, just as before, we may drop the subscript $k$ for convenience.
\begin{lem}\label{lemN}
If $N\geq5$, then for any $(\varphi,\psi)$ satisfying $|\varphi|\leq\frac12 U_1(y), |\psi|\leq\frac12 V_1(y)$
there holds
\begin{align*}
\|N_k(\varphi,\psi)\|_{**}\leq C\|(\varphi,\psi)\|_*^{1+\delta},
\end{align*}
where $0<\delta<\frac{4}{N-2}$.
\end{lem}

\begin{proof}
By definition of $N=N_k$ in \eqref{N}, observing that for the terms $N^{1,1}$ and
$N^{2,1}$, we just follow the estimates as in Lemma 2.4 of \cite{Wei-Yan} to get that, for $i=1,2$,
\begin{align*}
\|N^{i,1}\|_{**}\leq C\|(\varphi,\psi)\|_*^{2^*-1}.
\end{align*}
Since for $p\geq\frac{N-2-m}{N-2}$,
\begin{align*}
\sum_{j=1}^k\frac{1}{(1+|y-x_j|)^{p}}\leq C+\sum_{j=2}^k\frac{1}{|x_j-x_1|^{p}}\leq C,
\end{align*}
combining H\"older inequalities, for $i=1,2$, we obtain
\begin{align}\label{2.4-1}
\begin{split}
|N^{i,1}|&\leq\|\varphi\|_*^{2^*-1}\Big(\sum_{j=1}^k\frac{\lambda^{\frac{N-2}{2}}}{(1+\lambda|y-x_j|)^{\frac{N-2}{2}+\tau}}\Big)^{2^*-1}\\
&\leq C\|\varphi\|_*^{2^*-1}\lambda^{\frac{N+2}{2}}\sum_{j=1}^k\frac{1}{(1+\lambda|y-x_j|)^{\frac{N+2}{2}+\tau}}\Big(\sum_{j=1}^k\frac{1}{(1+\lambda|y-x_j|)^{\tau}}\Big)^{\frac4{N-2}}\\
&\leq C\|(\varphi,\psi)\|_*^{2^*-1}\lambda^{\frac{N+2}{2}}\sum_{j=1}^k\frac{1}{(1+\lambda|y-x_j|)^{\frac{N+2}{2}+\tau}}.
\end{split}
\end{align}

Next for the terms $N^{1,2}$ and $N^{2,2}$ in \eqref{N},
it suffices to deal with $N^{1,2}$, since the other one can be estimated in the same way.

 Rewrite
\begin{align}\label{2.4-2}
 N^{1,2}=&(W_{1,\lambda}+\varphi)^{\frac{2}{N-2}}(W_{2,\lambda}+\psi)^{\frac{N}{N-2}}-W_{1,\lambda}^{\frac{2}{N-2}}W_{2,\lambda}^{\frac{N}{N-2}}
-\frac 2{N-2} W_{1,\lambda}^{\frac{4-N}{N-2}}W_{2,\lambda}^{\frac{N}{N-2}}\varphi\cr
&-\frac N{N-2} W_{1,\lambda}^{\frac{2}{N-2}}W_{2,\lambda}^{\frac{2}{N-2}}\psi
=I+II+III,
\end{align}

where
\begin{align*}
&I= (W_{1,\lambda}+\varphi)^{\frac{2}{N-2}}\Big((W_{2,\lambda}+\psi)^{\frac{N}{N-2}}-W_{2,\lambda}^{\frac{N}{N-2}}-\frac N{N-2}W_{2,\lambda}^{\frac{2}{N-2}}\psi\Big),\\
&II=\Big((W_{1,\lambda}+\varphi)^{\frac{2}{N-2}}-W_{1,\lambda}^{\frac{2}{N-2}}-\frac 2{N-2} W_{1,\lambda}^{\frac{4-N}{N-2}}\varphi
\Big)W_{2,\lambda}^{\frac{N}{N-2}},\\
&III=
\frac N{N-2}\Big( (W_{1,\lambda}+\varphi)^{\frac{2}{N-2}}-W_{1,\lambda}^{\frac{2}{N-2}}\Big)W_{2,\lambda}^{\frac{2}{N-2}}\psi.
\end{align*}
First for $I$, we have
\begin{align*}
|I|\leq C\Bigl(|\psi|^{\frac{N}{N-2}}W_{1,\lambda}^{\frac{2}{N-2}}+|\psi|^{\frac{N}{N-2}}|\varphi|^{\frac{2}{N-2}}\Bigr):=I_1+I_2.
\end{align*}
We estimate by H\"older inequalities,
\begin{align*}
|I_1|\leq& C\|\psi\|_*^{\frac{N}{N-2}}\Big(\sum_{j=1}^k\frac{\lambda^{\frac{N-2}{2}}}{(1+\lambda|y-x_j|)^{\frac{N-2}{2}+\tau}}\Big)^{\frac N{N-2}}
\Big(\sum_{j=1}^k\frac{\lambda^{\frac{N-2}{2}}}{(1+\lambda|y-x_j|)^{N-2}}\Big)^{\frac{2}{N-2}}\cr
\leq &C \lambda^{\frac{N+2}{2}}\|\psi\|_*^{\frac{N}{N-2}}\Big(\sum_{j=1}^k\frac{1}{(1+\lambda|y-x_j|)^{\frac{N-2}{2}+\tau}}\Big)^{2^*-1}\cr
\leq & C\lambda^{\frac{N+2}{2}}\|\psi\|_*^{\frac{N}{N-2}}\sum_{j=1}^k\frac{1}{(1+\lambda|y-x_j|)^{\frac{N+2}{2}+\tau}}\Big(\sum_{j=1}^K\frac{1}{(1+\lambda|y-x_j|)^{\tau}}\Big)^{\frac4{N-2}}\cr
\leq& C\lambda^{\frac{N+2}{2}}\|(\varphi,\psi)\|_*^{\frac{N}{N-2}}\sum_{j=1}^k\frac{1}{(1+\lambda|y-x_j|)^{\frac{N+2}{2}+\tau}}.
\end{align*}
Moreover, we have
\begin{align*}
|I_2|&\leq C\|\psi\|_*^{\frac{N}{N-2}}\|\varphi\|_*^{\frac{2}{N-2}}\Big(\sum_{j=1}^k\frac{\lambda^{\frac{N-2}{2}}}{(1+\lambda|y-x_j|)^{\frac{N-2}{2}+\tau}}\Big)^{\frac {N+2}{N-2}}\cr
&\leq C\lambda^{\frac{N+2}{2}}\|\psi\|_*^{\frac{N}{N-2}}\|\varphi\|_*^{\frac{2}{N-2}}\sum_{j=1}^k\frac{1}{(1+\lambda|y-x_j|)^{\frac{N+2}{2}+\tau}}
\Big(\sum_{j=1}^k\frac{1}{(1+\lambda|y-x_j|)^{\tau}}\Big)^{\frac {4}{N-2}}\cr
&\leq C\lambda^{\frac{N+2}{2}}\|\psi\|_*^{\frac{N}{N-2}}\|\varphi\|_*^{\frac{2}{N-2}}\sum_{j=1}^k\frac{1}{(1+\lambda|y-x_j|)^{\frac{N+2}{2}+\tau}}.
\end{align*}
Therefore, we get
\begin{align*}
|I|\leq C\lambda^{\frac{N+2}{2}}\Bigl(\|\psi\|_*^{\frac{N}{N-2}}+\|\psi\|_*^{\frac{N}{N-2}}\|\varphi\|_*^{\frac{2}{N-2}}\Bigr)\sum_{j=1}^k\frac{1}{(1+\lambda|y-x_j|)^{\frac{N+2}{2}+\tau}}.
\end{align*}
For $II$, we will deal with it by dividing into the following two cases.
It follows from Lemma \ref{lmA.6} that $|\frac{\varphi_i}{W_i}|<\frac{1}{2}.$ Using the basic inequality
\begin{align}
(1+t)^{\frac{2}{N-2}}-1-\frac{2}{N-2}t=O(t^{2}),\,\,\,\text{for}\,\,\,|t|<\frac12,
\end{align}
we obtain if $N=5$ by using H\"older again,
\begin{align}\label{2.4-4}
 |II|=&\Big|\Big((W_{1,\lambda}+\varphi)^{\frac{2}{N-2}}-W_{1,\lambda}^{\frac{2}{N-2}}-\frac 2{N-2} W_{1,\lambda}^{\frac{4-N}{N-2}}\varphi
\Big)W_{2,\lambda}^{\frac{N}{N-2}}\Big|\cr
=&W_{1,\lambda}^{\frac{2}{N-2}}W_{2,\lambda}^{\frac{N}{N-2}}\Big|\Bigl(1+\frac{\varphi}{W_{1,\lambda}}\Bigr)^{\frac{2}{N-2}}-1-\frac 2{N-2} \frac{\varphi}{W_{1,\lambda}}
\Big|
\leq C W_{1,\lambda}^{\frac{2}{N-2}-2}W_{2,\lambda}^{\frac{N}{N-2}}|\varphi|^2\cr
\leq& C \|\varphi\|_*^{2}\lambda^{\frac{N+2}{2}}\Big(\sum_{j=1}^k\frac{1}{(1+\lambda|y-x_j|)^{N-2}}\Big)^{\frac{6-N}{N-2}}\Big(\sum_{j=1}^k\frac{1}{(1+\lambda|y-x_j|)^{\frac{N-2}{2}+\tau}}\Big)^{2}\cr
\leq&  C \lambda^{\frac{N+2}{2}}\|\varphi\|_*^{2}\Big(\sum_{j=1}^k\frac{1}{(1+\lambda|y-x_j|)^{\frac{N-2}{2}+\tau}}\Big)^{\frac {N+2}{N-2}}\cr
\leq&  C \lambda^{\frac{N+2}{2}}\|\varphi\|_*^{2}
\sum_{j=1}^k\frac{1}{(1+\lambda|y-x_j|)^{\frac{N+2}{2}+\tau}}.
\end{align}

Also for $|t|<1$ and $N\geq 6,$ there holds
$(1+t)^{\frac{2}{N-2}}-1-\frac{2}{N-2}t=O(t^{1+\delta}),$ where $\frac{4}{N-2}-\delta>0$.
Hence when $N\geq 6,$ we have
\begin{align*}
\begin{split}
 |II|&=\Big|\Big((W_{1,\lambda}+\varphi)^{\frac{2}{N-2}}-W_{1,\lambda}^{\frac{2}{N-2}}-\frac 2{N-2} W_{1,\lambda}^{\frac{4-N}{N-2}}\varphi
\Big)W_{2,\lambda}^{\frac{N}{N-2}}\Big|\cr
&=W_{1,\lambda}^{\frac{2}{N-2}}W_{2,\lambda}^{\frac{N}{N-2}}\Big|\Bigl(1+\frac{\varphi}{W_{1,\lambda}}\Bigr)^{\frac{2}{N-2}}-1-\frac 2{N-2} \frac{\varphi}{W_{1,\lambda}}
\Big|
\leq C W_{1,\lambda}^{\frac{2}{N-2}}W_{2,\lambda}^{\frac{N}{N-2}}\Big|\frac{\varphi}{W_{1,\lambda}}\Big|^{1+\delta}\cr
&\leq C\|\varphi\|_{*}^{1+\delta}\Big(\sum\limits_{j=1}^k\frac{\lambda^{\frac{N-2}{2}}}{(1+\lambda|y-x_j|)^{N-2}}\Big)^{\frac{4}{N-2}-\delta}\Big(\sum\limits_{j=1}^k\frac{\lambda^{\frac{N-2}{2}}}
{(1+\lambda|y-x_j|)^{\frac{N-2}{2}+\tau}}\Big)^{1+\delta}\cr
&\leq C\lambda^{\frac{N+2}{2}}\|\varphi\|_{*}^{1+\delta}\Big(\sum\limits_{j=1}^k\frac{1}{(1+\lambda|y-x_j|)^{\frac{N-2}{2}+\tau}}\Big)^{\frac{N+2}{N-2}}
\leq C\lambda^{\frac{N+2}{2}}\|\varphi\|_{*}^{1+\delta}\sum\limits_{j=1}^k\frac{1}{(1+\lambda|y-x_j|)^{\frac{N+2}{2}+\tau}},
\end{split}
\end{align*}
since
\begin{align*}
\Big(\sum\limits_{j=1}^k\frac{1}{(1+\lambda|y-x_j|)^{\bar{\sigma}}}\Big)^{\frac{N+2}{N-2}}
&\leq\Big(\sum\limits_{j=1}^k\frac{1}{(1+\lambda|y-x_j|)^{\frac{N+2}{2}+\tau}}\Big)\Big(\sum\limits_{j=1}^k\frac{1}{(1+\lambda|y-x_j|)^\tau}\Big)^{\frac{4}{N-2}}\cr
&\leq C\Bigl(\frac{1}{\lambda}\Bigr)^{1+\epsilon}\sum\limits_{j=1}^k\frac{1}{(1+\lambda|y-x_j|)^{\frac{N+2}{2}+\tau}},
\end{align*}
which means that
\begin{align}\label{2.4-4''}
\begin{split}
 |II|\leq C \lambda^{\frac{N+2}{2}}\|\varphi\|_*^{1+\delta}\sum_{j=1}^k\frac{1}{(1+\lambda|y-x_j|)^{\frac{N+2}{2}+\tau}}.
\end{split}
\end{align}

Finally, for $III$, we have
\begin{align}\label{2.4-5}
\begin{split}
 |III|&\leq C
|\varphi|^{\frac{2}{N-2}}|\psi|W_{2,\lambda}^{\frac{2}{N-2}}\cr
&\leq C\|\varphi\|_*^{\frac{2}{N-2}}\|\psi\|_*\Big(\sum_{j=1}^k\frac{\lambda^{\frac{N-2}{2}}}{(1+\lambda|y-x_j|)^{\frac{N-2}{2}+\tau}}\Big)^{\frac{N}{N-2}}
\Big(\sum_{j=1}^k\frac{\lambda^{\frac{N-2}{2}}}{(1+\lambda|y-x_j|)^{N-2}}\Big)^{\frac{2}{N-2}}\cr
&\leq C\lambda^{\frac{N+2}{2}}\|\varphi\|_*^{\frac{2}{N-2}}\|\psi\|_*\Big(\sum_{j=1}^k\frac{1}{(1+\lambda|y-x_j|)^{\frac{N-2}{2}+\tau}}\Big)^{\frac {N+2}{N-2}}\cr
&\leq  C\lambda^{\frac{N+2}{2}} \|\varphi\|_*^{\frac{2}{N-2}}\|\psi\|_*
\sum_{j=1}^k\frac{1}{(1+\lambda|y-x_j|)^{\frac{N+2}{2}+\tau}}.
\end{split}
\end{align}

From \eqref{2.4-1} to \eqref{2.4-5}, we obtain that
\begin{align*}
\|N(\varphi,\psi)\|_{**}\leq C\|(\varphi,\psi)\|_*^{1+\delta}.
\end{align*}
\end{proof}
Next, we estimate $R_k$.
\begin{lem}\label{lemR}
Then, there exists some small $\epsilon>0$ such that
\begin{align}\label{RK}
\|R_k\|_{**}\leq C\left(\frac1{\lambda}\right)^{1+\epsilon}.
\end{align}
\end{lem}

\begin{proof}
We define
\begin{align*}
\Omega_j=\left\{y:y=(y',y'')\in\R^2\times\R^{N-2},\langle\frac{y'}{|y'|},\frac{x_j}{|x_j|}\rangle\geq \cos\frac\pi k\right\}.
\end{align*}
Recall by \eqref{N} that
\begin{small}
\begin{align}\label{2.5-0}
\begin{split}
 R_{1,K}&=
\Bigl(W_{1,\lambda}^{2^*-1}-\sum_{j=1}^kW_{1,x_j}^{2^*-1}\Bigr)-P(r,y'')W_{1,\lambda}+W^*_{1,\lambda}\Delta\xi+2\nabla\xi\nabla W_{1,\lambda}^*
+\frac{\beta}{2}\Big(W_{1,\lambda}^{\frac{2^*}{2}-1}W_{2,\lambda}^{\frac{2^*}{2}}-\sum_{j=1}^kW_{1,x_j}^{\frac{2^*}{2}-1}W_{2,x_j}^{\frac{2^*}{2}}\Big)
\\
&:=J_1+J_2+J_3+J_4+J_5.
\end{split}
\end{align}
\end{small}
By symmetry, we might as well assume that $y\in\Omega_1$, then $|y-x_j|\geq |y-x_1|$.
Therefore,
\begin{align}\label{2.5-1}
\begin{split}
|J_1|\leq &CW_{1,x_1}^{\frac{4}{N-2}}\sum\limits_{j=2}^kW_{1,x_j}+C\Bigl(\sum\limits_{j=2}^kW_{1,x_j}\Bigr)^{2^*-1}\cr
\leq &C\frac{\lambda^2}{(1+\lambda|y-x_1|)^{4}}\sum_{j=2}^k\frac{\lambda^{\frac{N-2}{2}}}{(1+\lambda|y-x_j|)^{N-2}}+C\Big(\sum_{j=2}^k\frac{\lambda^{\frac{N-2}{2}}}{(1+\lambda|y-x_j|)^{N-2}}\Big)^{2^*-1}.
\end{split}
\end{align}
By Lemma \ref{lmA.3}, taking $0<\alpha\leq\min\{4,N-2\}$, we obtain that for any $y\in \Omega_1$ and $j>1$, then
\begin{align}\label{2.5-2}
\begin{split}
\frac{1}{(1+\lambda|y-x_1|)^{4}}\sum_{j=2}^k\frac{1}{(1+\lambda|y-x_j|)^{N-2}}
\leq C\frac{1}{(1+\lambda|y-x_1|)^{N+2-\alpha}}\sum_{j=2}^k\frac{1}{|\lambda(x_j-x_1)|^{\alpha}}.
\end{split}
\end{align}
We choose $\alpha>\frac{N}{2}$ satisfying $\frac{N-2}2\leq\alpha\leq4-\tau$ to obtain that                                                               
\begin{align*}
&\frac{1}{(1+\lambda|y-x_1|)^{4}}\sum_{j=2}^k\frac{1}{(1+\lambda |y-x_j|)^{N-2}}
\leq\frac{1}{(1+\lambda|y-x_1|)^{N-2+\tau}}\Bigl(\frac{k}{\lambda}\Bigr)^\alpha
\leq\frac1{C}\frac{1}{(1+\lambda|y-x_1|)^{\frac{N+2}{2}+\tau}}\Bigl(\frac{1}{\lambda}\Bigr)^{1+\epsilon}.
\end{align*}
Next, as in \cite{Peng-Wang-Yan}, for $y\in \Omega_1$, we can derive
\begin{align}\label{eqs2.34}
\Bigl(\sum\limits_{j=2}^k\frac{1}{1+\lambda|y-x_j|^{N-2}}\Bigr)^{2^*-1}
\leq&\sum\limits_{j=2}^k\frac{1}{1+\lambda|y-x_j|^{\frac{N+2}{2}+\tau}}\Bigl(\sum\limits_{j=2}^k\frac{1}{(1+\lambda|y-x_j|)^{\frac{N+2}{4}(\frac{N-2}{2}-\tau(\frac{N-2}{N+2}))}}\Bigr)^{\frac{4}{N-2}}\cr
\leq&\Bigl(\frac{k}{\lambda}\Bigr)^{\frac{N+2}{4}\Bigl(\frac{N-2}{2}-\tau(\frac{N-2}{N+2})\Bigr)}\sum\limits_{j=2}^k\frac{1}{(1+\lambda|y-x_j|)^{\frac{N+2}{2}+\tau}}\cr
\leq&\Bigl(\frac{1}{\lambda}\Bigr)^{1+\epsilon}\sum\limits_{j=2}^k\frac{1}{(1+\lambda|y-x_j|)^{\frac{N+2}{2}+\tau}}.
\end{align}
Hence \eqref{2.5-1}-\eqref{eqs2.34} imply that
\begin{align*}
\|J_1\|_{**}\leq \Bigl(\frac{1}{\lambda}\Bigr)^{1+\epsilon}.
\end{align*}

On the other hand, we obtain
\begin{align}\label{J2}
|J_2|\leq &C\sum\limits_{j=1}^k\frac{\lambda^{\frac{N-2}{2}}\xi}{(1+\lambda|y-x_j|)^{N-2}}
\leq C\lambda^{\frac{N+2}{2}}\sum\limits_{j=1}^k\frac{\xi}{\lambda^2(1+\lambda|y-x_j|)^{N-2}}\cr
\leq &C\Bigl(\frac{1}{\lambda}\Bigr)^{1+\epsilon}\sum\limits_{j=1}^k\frac{\lambda^{\frac{N+2}{2}}\xi}{\lambda^{1-\epsilon}(1+\lambda|y-x_j|)^{N-2}}
\leq C\Bigl(\frac{1}{\lambda}\Bigr)^{1+\epsilon}\sum\limits_{j=1}^k\frac{\lambda^{\frac{N+2}{2}}}{(1+\lambda|y-x_j|)^{\frac{N+2}{2}+\tau}},
\end{align}
since $N-1-\epsilon>\frac{N+2}{2}+\tau$ if $\epsilon>0$ is small. Hence $\|J_2\|_{**}\leq C\Bigl(\frac{1}{\lambda}\Bigr)^{1+\epsilon}$.

Similar to \eqref{J2}, we have
\begin{align}\label{J3}
|J_3|\leq& C\sum\limits_{j=1}^k\frac{\lambda^{\frac{N-2}{2}}|\Delta\xi|}{(1+\lambda|y-x_j|)^{N-2}}\leq C\lambda^{\frac{N+2}{2}}\sum\limits_{j=1}^k\frac{|\Delta\xi|}{\lambda^2(1+\lambda|y-x_j|)^{N-2}}\cr
\leq& C\Bigl(\frac{1}{\lambda}\Bigr)^{1+\epsilon}\sum\limits_{j=1}^k\frac{|\Delta \xi|\lambda^{\frac{N+2}{2}}}{\lambda^{1-\epsilon}(1+\lambda|y-x_j|)^{N-2}}
\leq C\Bigl(\frac{1}{\lambda}\Bigr)^{1+\epsilon}\sum\limits_{j=1}^k\frac{\lambda^{\frac{N+2}{2}}}{(1+\lambda|y-x_j|)^{\frac{N+2}{2}+\tau}},
\end{align}
which yields that $\|J_3\|_{**}\leq C\Bigl(\frac{1}{\lambda}\Bigr)^{1+\epsilon}$.

Moreover, we can check that
\begin{align}\label{J4}
|J_4|\leq& C\sum\limits_{j=1}^k\frac{\lambda^{\frac{N}{2}}|\nabla\xi|}{(1+\lambda|y-x_j|)^{N-1}}\leq C\lambda^{\frac{N+2}{2}}\sum\limits_{j=1}^k\frac{|\nabla\xi|}{\lambda(1+\lambda|y-x_j|)^{N-1}}\cr
\leq&\Bigl(\frac{1}{\lambda}\Bigr)^{1+\epsilon}\sum\limits_{j=1}^k\frac{|\Delta\xi|\lambda^{\frac{N+2}{2}}}{\lambda^{-\epsilon}(1+\lambda|y-x_j|)^{N-2}}\cr
\leq&C\Bigl(\frac{1}{\lambda}\Bigr)^{1+\epsilon}\sum\limits_{j=1}^k\frac{\lambda^{\frac{N+2}{2}}}{(1+\lambda|y-x_j|)^{\frac{N+2}{2}+\tau}},
\end{align}
since $\delta\leq|(r,y'')-(r_0,y_0'')|\leq2\delta$ and \eqref{weqs1.3} implies that
\begin{align*}
\frac{1}{1+\lambda|y-x_j|}\leq\frac{1}{\lambda}.
\end{align*}
Hence, we obtain
$\|J_4\|_{**}\leq C\Bigl(\frac{1}{\lambda}\Bigr)^{1+\epsilon}$.

As a result, from \eqref{2.5-1} to \eqref{J4}, we have
\begin{align*}
\|l_k\|_{**}\leq C\Bigl(\frac{1}{\lambda}\Bigr)^{1+\epsilon}
\end{align*}
\end{proof}
Now, we are ready to prove Proposition \ref{propnonlinear}.

\begin{proof}[
\textbf{Proof of Proposition \ref{propnonlinear}:}]
Recall that $\lambda=k^{\frac{N-2}{N-4}}$. Let
\begin{align*}
E=\Big\{(\varphi_{1},\varphi_{2})\in &(C(\R^N)\cap H_s)^2,\  \|(\varphi,\psi)\|_*\leq\frac1{\lambda},\
|\varphi_i|\leq\frac12 W_i,\ i=1,2,\\
&\big\langle (w_{1,x_j}^{2^*-2}Y_{j,l},w_{2,x_j}^{2^*-2}Z_{j,l}), (\varphi,\psi)\big\rangle=0,\ j=1,\ldots,k,\ l=1,2
\Big\}.
\end{align*}
Then we are sufficed to solve
\begin{align*}
(\varphi,\psi)=A(\varphi,\psi):=\mathbb L_k(N(\varphi,\psi))+\mathbb L_k(R_k)
\end{align*}
with $\mathbb L_k$ defined as in Proposition \ref{proplinear}.
We will prove that $A$ is a contraction map from $E$ to $E$.

\medskip
First, by   Proposition \ref{proplinear}, Lemma \ref{lemN},  Lemma \ref{lemR},  $A$ maps $E$ to $E$ and
\begin{align*}
 \|A(\varphi,\psi)\|_*&
 \leq C\|N(\varphi,\psi)\|_{**}+C\|R_k\|_{**}\\
 &\leq  C\|(\varphi,\psi)\|_*^{1+\delta}+\Bigl(\frac {1}{\lambda}\Bigr)^{1+\epsilon}
 \leq \frac{1}{\lambda}.
\end{align*}

\medskip
Next it is obvious
\begin{align*}
 \|A(\varphi,\psi)-A(\tilde\varphi,\tilde\psi)\|_*&
 \leq C\|N(\varphi,\psi)-N(\tilde\varphi,\tilde\psi)\|_{**}.
\end{align*}
Recall the  definition of \eqref{N}.
We first deal with $N_K^{i,2}\ (i=1,2)$
by considering $$f(t,s):=(1+t)^{\frac{2}{N-2}}(1+s)^{\frac{N}{N-2}}-1
-\frac 2{N-2} t-\frac N{N-2} s.$$
It is easy to have
\begin{align*}
&f'_t(t,s)=\frac{2}{N-2}(1+t)^{\frac{4-N}{N-2}}(1+s)^{\frac{N}{N-2}}-\frac{2}{N-2},\\&
f'_s(t,s)=\frac{N}{N-2}(1+t)^{\frac{2}{N-2}}(1+s)^{\frac{2}{N-2}}-\frac{N}{N-2}.
\end{align*}
 For the case $t+1,s+1>\frac12$, $|f'_t(t,s)|$ and $|f'_s(t,s)|=O(|s|+|t|)$, so we have
\begin{align*}
&|N^{i,2}(\varphi,\psi)-N^{i,2}(\tilde\varphi,\tilde\psi)|\\
&\leq
 CW_{1,\lambda}^{\frac{2}{N-2}}W_{2,\lambda}^{\frac{N}{N-2}}\Bigl(\Bigl|\frac{\varphi}{W_{1,\lambda}}\Bigr|+\Bigl|\frac{\psi}{W_{2,\lambda}}\Bigr|\Bigr)\Bigl(\frac{|\varphi-\tilde\varphi|}{W_{1,\lambda}}
+\frac{|\psi-\tilde\psi|}{W_{2,\lambda}}\Bigr)\\
&\leq C\Big(\sum_{j=1}^k\frac{1}{(1+|y-x_j|)^{N-2}}\Big)^{\frac{4}{N-2}}\Big(\sum_{j=1}^k\frac{1}{(1+|y-x_j|)^{\bar\sigma}}\Big)^{2}  \\&\quad \times (\|\varphi\|_*+\|\psi\|_*)(\|\varphi-\tilde\varphi\|_*
+\|\psi-\tilde\psi\|_*)
 \\
&\leq (\|\varphi\|_*+\|\psi\|_*)(\|\varphi-\tilde\varphi\|_*
+\|\psi-\tilde\psi\|_*)
 \Big(\sum_{j=1}^k\frac{1}{(1+|y-x_j|)^{\bar\sigma}}\Big)^{\frac{N+2}{N-2}}\\
&\leq\frac12(\|\varphi-\tilde\varphi\|_*+\|\psi-\tilde\psi\|_*).
\end{align*}
We can deal with $N^{1,1}$ and $N^{2,1}$ similarly.

\medskip
Therefore, $A$ is a contraction map and it follows from the contraction mapping theorem that there exists a unique $(\varphi,\psi)\in E,$
such that $$(\varphi,\psi)=A(\varphi,\psi).$$
Moreover, from Proposition \ref{proplinear} 
we get
$$\|(\varphi,\psi)\|_*\leq C\Big(\frac1\lambda\Big)^{1+\epsilon}.$$
\end{proof}

\section{Proof of the main result}\label{s3}

Let \begin{align*}
F(r,\lambda)=I(W_{1,\lambda}+\varphi,W_{2,\lambda}+\psi),
\end{align*}
where $r=|x_1|$, $(\varphi_1,\varphi_2)$ is obtained in Proposition \ref{propnonlinear},
and
\begin{align*}
I(u,v):=\frac12\int_{\R^N}(|\nabla u|^2+P(r,y'')u^2+|\nabla v|^2+Q(r,y'')v^2)
-\frac1{2^*}\int_{\R^N}\left(| u|^{2^*}+|v|^{2^*}+\beta|u|^{\frac{2^*}{2}}|v|^{\frac{2^*}{2}}\right).
\end{align*}
In this section, we will choose suitable $(\bar{r},\bar{y}'',\lambda)$ so that $(W_{1,\lambda}+\varphi,W_{2,\lambda}+\psi)$ is a solution of \eqref{eqs1.1}. For this purpose, we need the following result.
\begin{lemma}
Suppose that $(r,y'',\lambda)$ satisfies
\begin{align}\label{weqs3.1}
&\int_{D_\rho}\Bigl(-\Delta u_k+P(r,y'')u_k-u_k^{2^*}-\frac{\beta}{2}u_k^{\frac{2^*}{2}-1}v_k^{\frac{2^*}{2}}\Bigr)(y,\nabla u_k)=0,\cr
&\int_{D_\rho}\Bigl(-\Delta v_k+Q(r,y'')v_k-v_k^{2^*}-\frac{\beta}{2}v_k^{\frac{2^*}{2}-1}u_k^{\frac{2^*}{2}}\Bigr)(y,\nabla v_k)=0,
\end{align}
\begin{align}\label{weqs3.2}
&\int_{D_\rho}\Bigl(-\Delta u_k+P(r,y'')u_k-u_k^{2^*}-\frac{\beta}{2}u_k^{\frac{2^*}{2}-1}v_k^{\frac{2^*}{2}}\Bigr)\frac{\partial u_k}{\partial y_i}=0,\cr
&\int_{D_\rho}\Bigl(-\Delta v_k+Q(r,y'')v_k-v_k^{2^*}-\frac{\beta}{2}v_k^{\frac{2^*}{2}-1}u_k^{\frac{2^*}{2}}\Bigr)\frac{\partial v_k}{\partial y_i}=0,
\end{align}
and
\begin{align}\label{weqs3.3}
&\int_{\R^N}\Bigl(-\Delta u_k+P(r,y'')u_k-u_k^{2^*}-\frac{\beta}{2}u_k^{\frac{2^*}{2}-1}v_k^{\frac{2^*}{2}}\Bigr)\frac{\partial u_k}{\partial \lambda}=0,\cr
&\int_{\R^N}\Bigl(-\Delta v_k+Q(r,y'')v_k-v_k^{2^*}-\frac{\beta}{2}v_k^{\frac{2^*}{2}-1}u_k^{\frac{2^*}{2}}\Bigr)\frac{\partial v_k}{\partial \lambda}=0,
\end{align}
where $u_k=W_{1,\lambda}+\varphi$, $v_k=W_{2,\lambda}+\psi$ and $D_\rho=\ds\{(r,y''):|(r,y'')-(r_0,y_0'')|\leq\rho\}$ with $\rho\in(2\delta,5\delta)$. Then $c_l=0,l=1,2,\cdots,N$.
\end{lemma}
\begin{proof}
Since $(W_{1,\lambda},W_{2,\lambda})=(0,0)$ in $\R^N \backslash D_\rho$, we see that if \eqref{weqs3.1}-\eqref{weqs3.3} hold, then
\begin{align}\label{weqs3.4}
\sum\limits_{i=1}^Nc_i\sum\limits_{j=1}^k(W_{1,x_j}^{2^*-2}Y_{j,i}u,W_{2,x_j}^{2^*-2}Z_{j,i}v)=(0,0),
\end{align}
for $(u,v)=\Bigl(\langle y,\nabla u_k\rangle,\langle y,\nabla v_k\rangle\Bigr)$, $(u,v)=\Bigl(\frac{\partial u_k}{\partial y_i},\frac{\partial v_k}{\partial y_i}\Bigr)$ and $(u,v)=\Bigl(\frac{\partial W_{1,\lambda}}{\partial \lambda},\frac{\partial W_{2,\lambda}}{\partial\lambda}\Bigr).$

For any function $f,\varphi\in H^1(\R^N)$, we have
\begin{align}\label{weqs3.5}
\int_{\R^N}f \frac{\partial\varphi}{\partial y_i}=-\int_{\R^N}\varphi\frac{\partial f}{\partial y_i}.
\end{align}
Using \eqref{weqs3.5}, we can prove that from \eqref{weqs3.4} that
\begin{align}\label{weqs3.6}
\sum\limits_{i=1}^Nc_i\sum\limits_{j=1}^k(W_{1,x_j}^{2^*-2}Y_{j,i}u,W_{2,x_j}^{2^*-2}Z_{j,i}v)(1+o(1))=(0,0)
\end{align}
holds for $(u,v)=\Bigl(\langle y,\nabla u_k\rangle,\langle y,\nabla v_k\rangle\Bigr)$, $(u,v)=\Bigl(\frac{\partial u_k}{\partial y_i},\frac{\partial v_k}{\partial y_i}\Bigr)$ and $(u,v)=\Bigl(\frac{\partial W_{1,\lambda}}{\partial \lambda},\frac{\partial W_{2,\lambda}}{\partial\lambda}\Bigr)$.

We have
\begin{align*}
\langle y, \nabla W_{i,\lambda}\rangle=\langle y',\nabla_{y'}W_{i,\lambda}\rangle+\langle y'',\nabla_{y''}W_{i,\lambda}\rangle,\,\,\,i=1,2.
\end{align*}
So we have
\begin{align}\label{weqs3.7}
&\sum\limits_{l=1}^Nc_l\sum\limits_{j=1}^k\int_{\R^N}(W_{1,x_j}^{2^*-2}Y_{j,l}\langle y, \nabla W_{1,\lambda}\rangle, W_{2,x_j}^{2^*-2}Z_{j,l}\langle y, \nabla W_{2,\lambda}\rangle)\cr
=&\sum\limits_{j=1}^k\int_{\R^N}(W_{1,x_j}^{2^*-2}Y_{j,2}\langle {y'}, \nabla_{y'} W_{1,\lambda}\rangle, W_{2,x_j}^{2^*-2}Z_{j,2}\langle y, \nabla W_{2,\lambda}\rangle)c_2+o(1)\sum\limits_{t\neq 2}|c_t|
\end{align}
and
\begin{align}\label{weqs3.8}
&\sum\limits_{l=1}^N\sum\limits_{j=1}^k\Bigl(W_{1,x_j}^{2^*-2}Y_{j,l}\frac{\partial W_{1,\lambda}}{\partial y_i},W_{2,x_j}^{2^*-2}Z_{j,l}\frac{\partial W_{2,\lambda}}{\partial y_i}\Bigr)\cr
=&\sum\limits_{j=1}^k\int_{\R^N}\Bigl(W_{1,x_j}^{2^*-2}Y_{j,l}\frac{\partial W_{1,\lambda}}{\partial y_i},W_{2,x_j}^{2^*-2}Z_{j,2}\Bigr)c_i+o(1)\sum\limits_{t\neq i}^{i=3,\cdots,N}|c_t|.
\end{align}
Since
\begin{align*}
\sum\limits_{j=1}^k\int_{\R^N}\Bigl(W_{1,x_j}Y_{j,2}\langle y',\nabla_{y'} W_{1,\lambda}\rangle,W_{2,x_j}Z_{j,2}\langle y',\nabla_{y'} W_{2,\lambda}\rangle\Bigr)=k(a_1+o(1))\lambda^2
\end{align*}
and
\begin{align*}
\sum\limits_{j=1}^k\int_{\R^N}\Bigl(W_{1,x_j}^{2^*-2}Y_{j,i}\frac{\partial W_{1,\lambda}}{\partial y_i},W_{2,x_j}^{2^*-2}Z_{j,i}\frac{\partial W_{2,\lambda}}{\partial y_i}\Bigr)=k(a_2+o(1))\lambda^2
\end{align*}
for some constants $a_1\neq0$ and $a_2\neq 0$, we find from \eqref{weqs3.7}-\eqref{weqs3.8}
\begin{align*}
c_i=o\Bigl(\frac{1}{\lambda^2}\Bigr)c_1,\,\,\,i=1,2,\cdots,N.
\end{align*}

Now, we have
\begin{align*}
0=\sum\limits_{l=1}^Nc_l\sum\limits_{j=1}^k\int_{\R^N}W_{1,x_j}^{2^*-2}Y_{j,l}\frac{\partial W_{1,\lambda}}{\partial\lambda}
=&\sum\limits_{j=1}^k\int_{\R^N}W_{1,x_j}^{2^*-2}Y_{j,1}\frac{\partial W_{1,\lambda}}{\partial\lambda}c_1+O\Bigl(\frac{k}{\lambda^2}c_1\Bigr)\cr
=&k(a_3+o(1))c_1+O\Bigl(\frac{k}{\lambda^2}\Bigr)c_1
\end{align*}
for some constant $a_3\neq0$. So $c_1=0$.
\end{proof}

\begin{lemma}
We have
\begin{align}\label{wqf3.1}
&\int_{\R^N}\Bigl(-\Delta u_k+P(r,y'')u_k-u_k^{2^*-1}-\frac{\beta}{2}u_k^{\frac{2^*}{2}-1}v_k^{\frac{2^*}{2}}\Bigr)\frac{\partial W_{1,\lambda}}{\partial \lambda}\cr
&+\int_{\R^N}\Bigl(-\Delta v_k+Q(r,y'')v_k-v_k^{2^*-1}-\frac{\beta}{2}v_k^{\frac{2^*}{2}-1}u_k^{\frac{2^*}{2}}\Bigr)\frac{\partial W_{2,\lambda}}{\partial \lambda}\cr
=&k\Bigl[-\frac{B_1}{\lambda^3}P(r,y'')-\frac{B_2}{\lambda^3}Q(r,y'')+\sum\limits_{j=1}^k\frac{(C_1+\beta C_2)(N-2)}{\lambda^{N-1}|x_1-x_j|^{N-2}}+O\Bigl(\frac{1}{\lambda^{3+\epsilon}}\Bigr)\Bigr],
\end{align}
where $B_1=\ds\int_{\R^N}U_{0,1}^2,\,B_2=\int_{\R^N}V_{0,1}^2$ and $C_1+\beta C_2>0$ are integer.
\end{lemma}
\begin{proof}
We have
\begin{align*}
\hbox{left the hand side of}\,\eqref{wqf3.1}=&\int_{\R^N}\Bigl(-\Delta W_{1,\lambda}+P(r,y'')W_{1,\lambda}-W_{1,\lambda}^{2^*-1}-\frac{\beta}{2}W_{1,\lambda}^{\frac{2^*}{2}-1}W_{2,\lambda}^{\frac{2^*}{2}}\Bigr)\frac{\partial W_{1,\lambda}}{\partial \lambda}\,\,\,\,\,\,\,\,\,\,\,\,\,\,\,\,\,\,\,\,\,\,\,\,\,\,\,\,\,\,\,\,\,\,\,\,\,\,\,\,\,\,\,\,\,\,\,\,\,\,\,\,\,\,\,\,\,\,\,\,\,\,\,\,\,\,\,\cr
&+\int_{\R^N}\Bigl(-\Delta W_{2,\lambda}+Q(r,y'')W_{2,\lambda}-W_{2,\lambda}^{2^*-1}-\frac{\beta}{2}W_{1,\lambda}^{\frac{2^*}{2}}W_{2,\lambda}^{\frac{2^*}{2}-1}\Bigr)\frac{\partial W_{2,\lambda}}{\partial \lambda}\cr
&+\int_{\R^N}\Bigl[\Bigl(-\Delta\varphi+P(r,y'')\varphi-(2^*-1)W_{1,\lambda}^{2^*-2}\varphi\Bigr)\frac{\partial W_{1,\lambda}}{\partial\lambda}\cr
&+\Bigl(-\Delta\psi+Q(r,y'')\psi-(2^*-1)W_{2,\lambda}^{2^*-2}\psi\Bigr)\frac{\partial W_{2,\lambda}}{\partial\lambda}\Bigr]\cr
&-\int_{\R^N}\Bigl[\Bigl((W_{1,\lambda}+\varphi)^{2^*-1}-W_{1,\lambda}^{2^*-1}-(2^*-1)W_{1,\lambda}^{2^*-2}\varphi\Bigr)\frac{\partial W_{1,\lambda}}{\partial\lambda}\cr
&-\Bigl((W_{2,\lambda}+\psi)^{2^*-1}-W_{2,\lambda}^{2^*-1}-(2^*-1)W_{2,\lambda}^{2^*-2}\psi\Bigr)\frac{\partial W_{2,\lambda}}{\partial\lambda}\Bigr]\cr
&+\frac{\beta}{2}\int_{\R^N}\Bigl[\Bigl((W_{1,\lambda}+\varphi)^{\frac{2^*}{2}-1}(W_{2,\lambda}+\psi)^{\frac{2^*}{2}}-W_{1,\lambda}^{\frac{2^*}{2}-1}W_{2,\lambda}^{\frac{2^*}{2}}\Bigr)\frac{\partial W_{1,\lambda}}{\partial\lambda}\cr
&+\Bigl((W_{1,\lambda}+\varphi)^{\frac{2^*}{2}}(W_{2,\lambda}+\psi)^{\frac{2^*}{2}-1}-W_{1,\lambda}^{\frac{2^*}{2}}W_{2,\lambda}^{\frac{2^*}{2}-1}\Bigr)\frac{\partial W_{2,\lambda}}{\partial\lambda}\Bigr]\cr
=&:I_1+I_2+I_3+I_4.
\end{align*}
By Lemma \ref{lmA.1}, we obtain
\begin{align}\label{weqs3.11}
I_1=k\Bigl(-\frac{B_1}{\lambda^3}P(r,y'')-\frac{B_2}{\lambda^3}Q(r,y'')+\sum\limits_{j=2}^k\frac{B_2}{\lambda^{N-1}|x_1-x_j|^{N-2}}\Bigr)+O\Bigl(\frac{1}{\lambda^{3+\lambda}}\Bigr).
\end{align}
Using \eqref{eqs2.10} and \eqref{eqs2.11}, we have
\begin{align}\begin{split}\label{weqs3.12}
I_2&=\int_{\R^N}\Bigl(-\Delta\varphi+P(r,y'')\varphi-(2^*-1)W_{1,\lambda}^{2^*-2}\varphi\Bigr)\frac{\partial W_{1,\lambda}}{\partial\lambda}\\
&\quad +\Bigl(-\Delta\psi+Q(r,y'')\psi-(2^*-1)W_{2,\lambda}^{2^*-2}\psi\Bigr)\frac{\partial W_{2,\lambda}}{\partial\lambda}\\
&=O\Bigl(\frac{\|(\varphi,\psi)\|_{*}}{\lambda^{2+\epsilon}}\Bigr)=O\Bigl(\frac{1}{\lambda^{3+\epsilon}}\Bigr).
\end{split}\end{align}
Supposing that $N\geq6$ and denoting $I_3=I_{31}+I_{32}$, then we have
\begin{align}\label{weqs3.13}
|I_{31}|
=&\Bigl|\int_{\R^N}\Bigl((W_{1,\lambda}+\varphi)^{2^*-1}-W_{1,\lambda}^{2^*-1}-(2^*-1)W_{1,\lambda}^{2^*-2}\varphi\Bigr)\frac{\partial W_{1,\lambda}}{\partial\lambda}\Bigr|\cr
\leq &C\int_{\R^N}W_{1,\lambda}^{2^*-3}\varphi^2\Bigl|\frac{\partial W_{1,\lambda}}{\partial\lambda}\Bigr|
\leq \frac{C}{\lambda}\int_{\R^N}\Bigl(\sum\limits_{j=1}^kW_{1,x_j}\Bigr)^{2^*-2}\varphi^2\cr
\leq&\frac{\|\varphi\|_{*}^2}{\lambda}\int_{\R^N}\Bigl(\sum\limits_{j=1}^kW_{1,x_j}\Bigr)^{2^*-2}\Bigl(\sum\limits_{i=1}^k\frac{\lambda^{\frac{N-2}{2}}}{1+\lambda|y-x_i|^{\frac{N-2}{2}+\tau}}\Bigr)^2\cr
=&O\Bigl(\frac{k}{\lambda^{3+\epsilon}}\Bigr).
\end{align}
Similarly, for $N=5$, we have
\begin{align}\label{weqs3.14}
|I_{31}|\leq \int_{\R^N}\Bigl(W_{1,\lambda}^{2^*-3}\varphi^2\Bigl|\frac{\partial W_{1,\lambda}}{\partial\lambda}\Bigr|+|\varphi|^{2^*}\frac{\partial W_{1,\lambda}}{\partial\lambda}\Bigr)=O\Bigl(\frac{k}{\lambda^{3+\epsilon}}\Bigr).
\end{align}
Similarly, we have
\begin{align}\label{weqs3.15}
|I_{32}|
=\int_{\R^N}\Bigl((W_{2,\lambda}+\psi)^{2^*-1}-W_{2,\lambda}^{2^*-1}-(2^*-1)W_{2,\lambda}^{2^*-2}\psi\Bigr)\frac{\partial W_{2,\lambda}}{\partial\lambda}
=O\Bigl(\frac{k}{\lambda^{3+\epsilon}}\Bigr),
\end{align}
then $|I_3|=|I_{31}+I_{32}|=O\Bigl(\frac{k}{\lambda^{3+\epsilon}}\Bigr)$.

On the other hand, we have
\begin{align}\label{weqs3.16}
I_{41}=&\int_{\R^N}\Bigl((W_{1,\lambda}+\varphi)^{\frac{2^*}{2}-1}(W_{2,\lambda}+\psi)^{\frac{2^*}{2}}-W_{1,\lambda}^{\frac{2^*}{2}}W_{2,\lambda}^{\frac{2^*}{2}}\Bigr)\frac{\partial W_{1,\lambda}}{\partial\lambda}\cr
\leq &C\int_{\R^N}\Bigl(W_{1,\lambda}^{2^*-2}\varphi+W_{2,\lambda}^{2^*-3}\varphi\psi+W_{1,\lambda}^{2^*-2}\psi\Bigr)\frac{\partial W_{1,\lambda}}{\partial\lambda}\cr
\leq&\frac{1}{\lambda}\int_{\R^N}\Bigl(\sum\limits_{j=1}^kW_{1,x_j}\Bigr)^{2^*-1}\varphi=O\Bigl(\frac{1}{\lambda^{3+\epsilon}}\Bigr).
\end{align}
Similarly, we can also estimate
\begin{align}\label{weqs3.17}
I_{42}=\int_{\R^N}\Bigl((W_{1,\lambda}+\varphi)^{\frac{2^*}{2}}(W_{2,\lambda}+\psi)^{\frac{2^*}{2}-1}-W_{1,\lambda}^{\frac{2^*}{2}}W_{2,\lambda}^{\frac{2^*}{2}-1}\Bigr)\frac{\partial W_{2,\lambda}}{\partial\lambda}=O\Bigl(\frac{1}{\lambda^{3+\epsilon}}\Bigr).
\end{align}
Combining \eqref{weqs3.11}-\eqref{weqs3.17}, we complete the proof.
\end{proof}
Next, we estimate \eqref{weqs3.1} and \eqref{weqs3.2}. Let us point out that the left hand side of \eqref{weqs3.2} is the local Pohozaev identities generating from translation.  We integrate by parts to find that \eqref{weqs3.1} is equivalent to
\begin{align}\label{eqs3.16}
\int_{D_\rho}\frac{\partial P(r,y'')}{\partial y_i}u_k^2=O\Bigl(\int_{\partial D_\rho}|\nabla\varphi|^2+\varphi^2+|\varphi|^{2^*}\Bigr),\,\,\,
\int_{D_\rho}\frac{\partial Q(r,y'')}{\partial y_i}v_k^2=O\Bigl(\int_{\partial D_\rho}|\nabla\psi|^2+\psi^2+|\psi|^{2^*}\Bigr),
\end{align}
since $(u_k,v_k)=(\varphi,\psi)$ on $\partial D_\rho$.

On the other hand, we know that \eqref{weqs3.1} is equivalent to
\begin{align}\label{eqs3.19}
&\frac{N-2}{2}\int_{\Omega}\Bigl(|\nabla u|^2+|\nabla v|^2\Bigr)-\frac{1}{2}\int_{\Omega}(NP(y)+\langle y,\nabla P\rangle)u^2+\frac{N}{2^*}\int_{\Omega}\Bigl(u^{2^*}+v^{2^*}+\beta u^{\frac{2^*}{2}}v^{\frac{2^*}{2}}\Bigr)\cr
&-\frac{1}{2}\int_{\Omega}(NQ(y)+\langle y,\nabla Q\rangle)v^2\cr
=&O\Bigl(\int_{\partial\Omega}|\nabla\varphi|^2+\varphi^2+|\varphi|^{2^*}+|\nabla\psi|^2+\psi^2+|\psi|^{2^*}\Bigr).
\end{align}
From \eqref{eqnonlinear0}, we obtain
\begin{align}\label{eqs3.21}
\int_{D_\rho}|\nabla u_k|^2+P(y)u_k^2=&\int_{D_\rho}u_k^{2^*}+\frac{\beta}{2}\int_{D_\rho}u_k^{\frac{2^*}{2}}v_k^{\frac{2^*}{2}}+\sum\limits_{l=1}^Nc_l\sum\limits_{j=1}^k\int_{\R^N}W_{1,x_j}^{2^*-2}Y_{j,l}u_k\cr
&+O\Bigl(\int_{\partial D_\rho}|\nabla \varphi|^2+\varphi^2\Bigr)
\end{align}
and
\begin{align}\label{eqs3.22}
\int_{D_\rho}|\nabla v_k|^2+Q(y)v_k^2=&\int_{D_\rho}v_k^{2^*}+\frac{\beta}{2}\int_{D_\rho}u_k^{\frac{2^*}{2}}v_k^{\frac{2^*}{2}}+\sum\limits_{l=1}^Nc_l\sum\limits_{j=1}^k\int_{\R^N}W_{2,x_j}^{2^*-2}Z_{j,l}u_k\cr
&+O\Bigl(\int_{\partial D_\rho}|\nabla \psi|^2+\psi^2\Bigr).
\end{align}
Using \eqref{eqs3.19}-\eqref{eqs3.22}, we can get the following estimate for $c_i$,
\begin{align}\label{eqs3.23}
c_i=O\Bigl(\frac{1}{\lambda^{3+\epsilon}}\Bigr),\,\,\,i=2,\cdots,N,\,\,\,c_1=o\Bigl(\frac{1}{\lambda}\Bigr).
\end{align}
Noting
\begin{align}\label{eqs3.24}
\int_{\R^N}W_{1,x_j}^{2^*-1}Y_{j,l}=O\Bigl(\frac{1}{\lambda^N}\Bigr),\,\,\,\int_{\R^N}W_{2,x_j}Z_{j,l}=O\Bigl(\frac{1}{\lambda^N}\Bigr),
\end{align}
we find from
\eqref{eqs3.23} and \eqref{eqs3.24} that
\begin{align}\label{eqs3.25}
&\int_{D_\rho}P(y)+\frac{1}{2}\langle y,\nabla P(y)\rangle u_k^2+Q(y)+\frac{1}{2}\langle y,\nabla Q(y)\rangle v_k^2\cr
=&O\Bigl(\frac{k}{\lambda^{2+\epsilon}}\Bigr)+O\Bigl(\int_{\partial D_\rho}|\nabla \varphi|^2+\varphi^2+|\varphi|^{2^*}+|\nabla \psi|^2+\psi^2+|\psi|^{2^*}\Bigr)
\end{align}
for some small $\epsilon>0$.

From \eqref{eqs3.16}, we can rewrite \eqref{eqs3.25},
\begin{align*}
&\int_{D_\rho}\Bigl(P(y)+\frac{1}{2}r\frac{\partial P(r,y'')}{\partial r}\Bigr)u_k^2+\Bigl(Q(y)+\frac{1}{2}r\frac{\partial Q(r,y'')}{\partial r}\Bigr)v_k^2\cr
=&O\Bigl(\frac{k}{\lambda^{2+\epsilon}}\Bigr)+O\Bigl(\int_{\partial D_\rho}|\nabla \varphi|^2+\varphi^2+|\varphi|^{2^*}+|\nabla \psi|^2+\psi^2+|\psi|^{2^*}\Bigr).
\end{align*}
That is
\begin{align*}
\int_{D_\rho}\frac{1}{2r}\Bigl(\frac{\partial r^2P(r,y'')}{\partial r}u_k^2+\frac{\partial r^2Q(r,y'')}{\partial r}v_k^2\Bigr)=O\Bigl(\frac{k}{\lambda^2}+\int_{\partial D_\rho}|\nabla \varphi|^2+\varphi^2+|\varphi|^{2^*}+|\nabla \psi|^2+\psi^2+|\psi|^{2^*}\Bigr).
\end{align*}
\begin{lemma}\label{lm3.3}
It holds
\begin{align*}
\int_{\R^N}|\nabla\varphi|^2+P(r,y'')|\varphi|^2+|\nabla\psi|^2+Q(r,y'')\psi^2=O\Bigl(\frac{k}{\lambda^{2+\epsilon}}\Bigr).
\end{align*}
\end{lemma}
\begin{proof}
By the system, we prove
\begin{align*}
&\int_{\R^N}|\nabla\varphi|^2+P(r,y'')\varphi^2+|\nabla\psi|^2+Q(r,y'')\psi^2\cr
=&\int_{\R^N}\Bigl[(W_{1,\lambda}+\varphi)^{2^*-1}+(W_{2,\lambda}+\psi)^{2^*-1}
+\frac{\beta}{2}\Bigl((W_{1,\lambda}+\varphi)^{\frac{2^*}{2}-1}(W_{2,\lambda}+\psi)^{\frac{2^*}{2}}\cr
&+(W_{1,\lambda}+\varphi)^{\frac{2^*}{2}}(W_{2,\lambda}+\psi)^{\frac{2^*}{2}-1}\Bigr)
-P(r,y'')W_{1,\lambda}-Q(r,y'')W_{2,\lambda}+\Delta W_{1,\lambda}\varphi+\Delta W_{2,\lambda}\psi\Bigr]\cr
=&\int_{\R^N}\Bigl((W_{1,\lambda}+\varphi)^{2^*-1}-W_{1,\lambda}^{2^*-1}\Bigr)\varphi+\Bigl((W_{2,\lambda}+\psi)^{2^*-1}-W_{2,\lambda}^{2^*-1}\Bigr)\psi\cr
&+\frac{\beta}{2}\int_{\R^N}\Bigl((W_{1,\lambda}+\varphi)^{\frac{2^*}{2}-1}(W_{2,\lambda}+\psi)^{\frac{2^*}{2}}\varphi-W_{1,\lambda}^{\frac{2^*}{2}-1}W_{2,\lambda}^{\frac{2^*}{2}}\varphi\cr
&+(W_{1,\lambda}+\varphi)^{\frac{2^*}{2}}(W_{2,\lambda}+\psi)^{\frac{2^*}{2}-1}\psi
-W_{1,\lambda}^{\frac{2^*}{2}}W_{2,\lambda}^{\frac{2^*}{2}-1}\psi\Bigr)\cr
&-\int_{\R^N} P(r,y'')W_{1,\lambda}\varphi+Q(r,y'')W_{2,\lambda}\psi\cr
&+\int_{\R^N}\Bigl(\Delta W_{1,\lambda}-W_{1,\lambda}^{2^*-1}-\frac{\beta}{2}W_{1,\lambda}^{\frac{2^*}{2}-1}W_{2,\lambda}^{\frac{2^*}{2}}\Bigr)\varphi
+\Bigl(\Delta W_{2,\lambda}-W_{2,\lambda}^{2^*-1}-\frac{\beta}{2}W_{1,\lambda}^{\frac{2^*}{2}}W_{2,\lambda}^{\frac{2^*}{2}-1}\Bigr)\psi\cr
=&:M_1+M_2+M_3+M_4.
\end{align*}
We have
\begin{align}\label{M1}
|M_1|\leq &C\int_{\R^N}W_{1,\lambda}^{2^*-2}\varphi^2+|\varphi|^{2^*}+W_{2,\lambda}^{2^*-2}\psi^2+|\psi|^{2^*}\cr
\leq& C\|(\varphi,\psi)\|_{*}^2\int_{\R^N}W_{1,\lambda}^{2^*-2}\Bigl(\sum\limits_{j=1}^k\frac{\lambda^{\frac{N-2}{2}}}{(1+\lambda|y-x_j|)^{\frac{N-2}{2}+\tau}}\Bigr)^2\cr
&+\|(\varphi,\psi)\|_{*}^{2^*}\int_{\R^N}\Bigl(\sum\limits_{j=1}^k\frac{\lambda^{\frac{N-2}{2}}}{(1+\lambda|y-x_j|)^{\frac{N-2}{2}+\tau}}\Bigr)^{2^*}\cr
\leq&C\Bigl(\|(\varphi,\psi)\|_{*}^2+\|(\varphi,\psi)\|_{*}^{2^*}\Bigr)\int_{\R^N}\Bigl(\sum\limits_{j=1}^k\frac{\lambda^{\frac{N-2}{2}}}{(1+\lambda|y-x_j|)^{\frac{N-2}{2}+\tau}}\Bigr)^{2^*}\cr
\leq& Ck\|(\varphi,\psi)\|_{*}^2\int_{\R^N}\Bigl(\frac{\lambda^{\frac{N-2}{2}}}{(1+\lambda|y-x_1|)^{\frac{N-2}{2}+\theta}}\Bigr)^{2^*}\Bigl(1+\sum\limits_{j=2}^k\frac{1}{(\lambda|x_j-x_1|)^{\tau-\theta}}\Bigr)^{2^*}
\cr
\leq&C k\|(\varphi,\psi)\|_{*}^2\lambda^{2^*\theta}\leq \frac{Ck}{\lambda^{2+2\epsilon-2^*\theta}},
\end{align}
where $\theta>0$ is any fixed small constant.
Similar as \eqref{M1}, for $W_{2,\lambda}=\kappa W_{1,\lambda}$, we obtain
\begin{align}\label{M2}
|M_2|\leq C k\|(\varphi,\psi)\|_{*}^2\lambda^{2^*\theta}\leq \frac{Ck}{\lambda^{2+2\epsilon-2^*\theta}}.
\end{align}
Similar to the proof of \eqref{2.5-1}, we can deduce
\begin{align}\label{M3}
|M_3|\leq & C\|(\varphi,\psi)\|_{*}\int_{\R^N}\sum\limits_{j=1}^m\frac{\lambda^{\frac{N-2}{2}}\xi}{(1+\lambda|y-x_j|)^{N-2}}\sum\limits_{i=1}^k\frac{\lambda^{\frac{N-2}{2}}}{(1+\lambda|y-x_i|)^{\frac{N-2}{2}+\tau}}\cr
\leq&C\frac{\|(\varphi,\psi)\|_{*}}{\lambda^{1+\epsilon}}\int_{\R^N}\sum\limits_{j=1}^k\frac{\lambda^{\frac{N+2}{2}}}{(1+\lambda|y-x_j|)^{\frac{N+2}{2}+\tau}}\sum\limits_{i=1}^k\frac{\lambda^{\frac{N-2}{2}}}{(1+\lambda|y-x_i|)^{\frac{N-2}{2}+\tau}}
\cr
\leq&\frac{\|(\varphi,\psi)\|_{*}}{\lambda^{1+\epsilon}}\leq\frac{C k}{\lambda^{2+2\epsilon}}.
\end{align}
Moreover, from the estimates of $J_1,\,J_3,\,J_4$ in Lemma \ref{lmA.5}, we obtain
\begin{align}\label{M4}
|M_4|\leq & C\frac{\|(\varphi,\psi)\|_{*}}{\lambda^{1+\epsilon}}\int_{\R^N}\sum\limits_{j=1}^k\frac{\lambda^{\frac{N+2}{2}}}{(1+\lambda|y-x_j|)^{\frac{N+2}{2}+\tau}}\sum\limits_{i=1}^k\frac{\lambda^{\frac{N-2}{2}}}{(1+\lambda|y-x_i|)^{\frac{N-2}{2}+\tau}}
\cr
\leq&C\frac{k\|(\varphi,\psi)\|_{*}}{\lambda^{1+\epsilon}}\leq C\frac{k}{\lambda^{2+2\epsilon}}.
\end{align}
Combining \eqref{M1}-\eqref{M4}, we deduce
\begin{align*}
\int_{\R^N}\Bigl(|\nabla\varphi|^2+P(r,y'')\varphi^2+|\nabla\psi|^2+Q(r,y'')\psi^2\Bigr)=O\Bigl(\frac{k}{\lambda^{2+\epsilon}}\Bigr).
\end{align*}
\end{proof}
\begin{lemma}
For any $C^1$ function $g(r,y'')$, it holds
\begin{align*}
\int_{D_\rho}g(r,y'')u_k^2=k\Bigl(\frac{1}{\lambda^2}g(r,y'')\int_{\R^N}U_{0,1}^2+o\Bigl(\frac{1}{\lambda^2}\Bigr)\Bigr).
\end{align*}
\end{lemma}
\begin{proof}
The proof is similar to Lemma 3.3 of \cite{Peng-Wang-Yan}, so we omit it.
\end{proof}
\begin{lemma}\label{lm3.5}
It holds
\begin{align*}
\int_{D_{4\delta}\setminus D_{3\delta}}(|\nabla\varphi|^2+|\nabla\psi|^2)=O\Bigl(\frac{k}{\lambda^{2+\epsilon}}\Bigr).
\end{align*}
\end{lemma}
\begin{proof}
Noting that $W_{1,\lambda}=0$ in $D_{5\delta}\setminus D_{2\delta}$, we find from \eqref{2.5-1} that
\begin{align*}
\int_{D_{4\delta}\setminus D_{3\delta}}(|\nabla\varphi|^2+|\nabla\psi|^2)\leq& C\int_{D_{5\delta}\setminus D_{2\delta}}(|\varphi|^2+|\psi|^2+|\varphi|^{2^*}+|\psi|^{2^*})\cr
\leq& C\int_{D_{5\delta}\setminus D_{3\delta}}(\varphi^2+\psi^2)=O\Bigl(\frac{k}{\lambda^{2+\epsilon}}\Bigr).
\end{align*}
\end{proof}
From Lemma \ref{lm3.3} and Lemma \ref{lm3.5}, we know
\begin{align*}
\int_{D_{4\delta}\setminus D_{3\delta}}\Bigl(|\nabla\varphi|^2+|\varphi|^2+|\varphi|^{2^*}+|\nabla\psi|^2+|\psi|^2+|\psi|^{2^*}\Bigr)=O\Bigl(\frac{k}{\lambda^{2+\epsilon}}\Bigr).
\end{align*}
As a result, we can find a $\rho\in (3\delta,4\delta)$, such that
\begin{align}\label{eqs3.28}
\int_{\partial D_\rho}\Bigl(|\nabla\varphi|^2+|\varphi|^2+|\varphi|^{2^*}+|\nabla\psi|^2+|\psi|^2+|\psi|^{2^*}\Bigr)=O\Bigl(\frac{k}{\lambda^{2+\epsilon}}\Bigr).
\end{align}
Now we apply Lemma 3.3 and \eqref{eqs3.28} to obtain
\begin{align*}
k\Bigl[\frac{s^2}{\lambda^2}\Bigl(\frac{\partial P(r,y'')}{\partial y_i}+\kappa^2\frac{\partial Q(r,y'')}{\partial y_i}\Bigr)\int_{\R^N}w_{0,1}^2+o\Bigl(\frac{1}{\lambda^2}\Bigr)\Bigr]=o\Bigl(\frac{k}{\lambda^2}\Bigr),
\end{align*}
and
\begin{align*}
k\Bigl[\frac{s^2}{\lambda^2}\frac{1}{2\bar{r}}\Bigl(\frac{\partial \bar{r}^2P(r,y'')}{\partial \bar{r}}+\kappa^2\frac{\partial\bar{r}^2 Q(r,y'')}{\partial \bar{r}}\Bigr)\int_{\R^N}w_{0,1}^2+o\Bigl(\frac{1}{\lambda^2}\Bigr)\Bigr]=o\Bigl(\frac{k}{\lambda^2}\Bigr).
\end{align*}
Therefore, the equations to determine $(\bar{r},y'')$ are
\begin{align*}
\frac{\partial( P(\bar{r},\bar{y}'')+\kappa^2Q(\bar{r},\bar{y}''))}{\partial \bar{y}_i}=o(1),\,\,\,\frac{\partial r^2P(\bar{r},\bar{y}'')+\kappa^2\partial r^2Q(\bar{r},\bar{y}'')}{\partial \bar{r}}=o(1),
\end{align*}
and
\begin{align}\label{B1}
-\frac{B_1(P(\bar{r},\bar{y}'')+\kappa^2Q(\bar{r},\bar{y}''))}{\lambda^3}+\frac{(C_1+\beta\kappa^{\frac{2^*}{2}} C_1)k^{N-2}}{\lambda^{N-1}}=O\Bigl(\frac{1}{\lambda^{3+\epsilon}}\Bigr).
\end{align}

Letting $\lambda=tk^{\frac{N-2}{N-4}}$, then $t\in [L_0,L_1]$ since $\lambda\in[L_0k^{\frac{N-2}{N-4}},L_1k^{\frac{N-2}{N-4}}]$. Then from \eqref{B1}, we get
\begin{align*}
-\frac{B_1(P(\bar{r},\bar{y}'')+\kappa^2Q(\bar{r},\bar{y}''))}{t^3}+\frac{(C_1+\beta \kappa^{\frac{2^*}{2}}C_1)k^{N-2}}{t^{N-1}}=o(1),\,\,\,t\in[L_0,L_1].
\end{align*}
Let
\begin{align}\label{B2}
F(t,\bar{r},\bar{y}'')=\Bigl(\nabla_{\bar{r},\bar{y}''}(P+\kappa^2Q)(\bar{r},\bar{y}''),-\frac{B_1}{t^3}+\frac{C_1+\beta \kappa^{\frac{2^*}{2}}C_1}{t^{N-1}}\Bigr),
\end{align}
then
\begin{align*}
deg\Bigl(F(t,\bar{r},\bar{y}''),[D_1,D_2]\times B_{\frac{1}{\mu^{1-\theta}}}(r_0,y_0'')\Bigr)=deg\Bigl(\nabla_{\bar{r},\bar{y}''}(P+\kappa^2Q)(\bar{r},\bar{y}''),B_{\frac{1}{\mu^{1-\theta}}}(r_0,y_0'')\Bigr)\neq0.
\end{align*}
So, \eqref{B2} have a solution $t_k\in [L_0,L_1]$, $(\bar{r}_k,\bar{y}'')\in B_\theta(r_0,y_0'')$.
\appendix

\section{energy estimates and some known results }\label{sa}

In this section, we mainly do some energy expansions and give some known results which are used before.

\begin{lemma}\label{lmA.1}
If $N\geq5$, then
\begin{align*}
I_1=k\Bigl(-\frac{B_1}{\lambda^3}(P(r,y'')+\kappa^2Q(r,y''))+\sum\limits_{j=2}^k\frac{C_1(1+\beta\kappa^{\frac{2^*}{2}})}{\lambda^{N-1}|x_1-x_j|^{N-2}}\Bigr)+O\Bigl(\frac{1}{\lambda^{3+\epsilon}}\Bigr),
\end{align*}
where $B_1=\ds\int_{\R^N}w_{0,1}^{2}$, $C_1=\ds\int_{\R^N}w_{0,1}^{2^*}$ and $1+\beta \kappa^{\frac{2^*}{2}}>0$.
\end{lemma}
\begin{proof}
We have
$$
\frac{\partial I(W_{1,x_j})}{\partial\lambda}=\frac{\partial W^*_{1,x_j}}{\partial\lambda}+O\Bigl(\frac{k}{\lambda^{N-2}}\Bigr).
$$
Direct calculations show
\begin{align*}
-\Delta\sum\limits_{j=1}^kW^*_{1,x_j}=\sum\limits_{j=1}^k{W^*}_{1,x_j}^{2^*-1}+\frac{\beta}{2}\sum\limits_{j=1}^k{W^*}_{1,x_j}^{\frac{2^*}{2}-1}{W^*}_{2,x_j}^{\frac{2^*}{2}}
\end{align*}
and
\begin{align*}
-\Delta\sum\limits_{j=1}^kW^*_{2,x_j}=\sum\limits_{j=1}^k{W^*}_{2,x_j}^{2^*-1}+\frac{\beta}{2}\sum\limits_{j=1}^k{W^*}_{1,x_j}^{\frac{2^*}{2}}{W^*}_{2,x_j}^{\frac{2^*}{2}-1}.
\end{align*}
Then, we have
\begin{align*}
I_1=&\int_{\R^N}P(y)W^*_{1,\lambda}\frac{\partial W^*_{1,\lambda}}{\partial\lambda}-\Bigl({W^*}_{1,\lambda}^{2^*-1}-\sum\limits_{j=1}^k{W^*}_{1,x_j}^{2^*-1}\Bigr)\frac{\partial {W^*}_{1,\lambda}}{\partial \lambda}
+Q(y)W^*_{2,\lambda}\frac{\partial W^*_{2,\lambda}}{\partial\lambda}\cr
&-\Bigl({W^*}_{2,\lambda}^{2^*-1}-\sum\limits_{j=1}^k{W^*}_{2,x_j}^{2^*-1}\Bigr)\frac{\partial {W^*}_{2,\lambda}}{\partial \lambda}
+\frac{\beta}{2}\int_{\R^N}\Bigl({W^*}_{1,\lambda}^{\frac{2^*}{2}-1}{W^*}_{2,\lambda}^{\frac{2^*}{2}}-\sum\limits_{j=1}^k{W^*}_{1,x_j}^{\frac{2^*}{2}-1}{W^*}_{2,x_j}^{\frac{2^*}{2}}\Bigr)\frac{\partial {W^*}_{1,\lambda}}{\partial\lambda}\cr
&+\Bigl({W^*}_{1,\lambda}^{\frac{2^*}{2}}{W^*}_{2,\lambda}^{\frac{2^*}{2}-1}-\sum\limits_{j=1}^k{W^*}_{1,x_j}^{\frac{2^*}{2}}{W^*}_{2,x_j}^{\frac{2^*}{2}-1}\Bigr)\frac{\partial {W^*}_{2,\lambda}}{\partial\lambda}.
\end{align*}
It is easy to check that
\begin{align*}
&\int_{\R^N}P(y)W^*_{1,\lambda}\frac{\partial W^*_{1,\lambda}}{\partial\lambda}=k\Bigl[\int_{\R^N}P(y)W^*_{1,x_1}\frac{\partial W^*_{1,x_1}}{\partial\lambda}+O\Bigl(\frac{1}{\lambda}\int_{\R^N}W^*_{1,x_1}\sum\limits_{j=2}^kmW^*_{1,x_j}\Bigr)\Bigr]\cr
=&k\Bigl[-\frac{P(\bar{r},\bar{y}'')}{\lambda^3}\int_{\R^N}U_{0,1}^2+O\Bigl(\frac{1}{\lambda^3}\sum\limits_{j=1}^k\frac{1}{(\lambda|x_1-x_j|)^{N-4}}\Bigr)+\frac{1}{\lambda^{2+\epsilon}}\Bigr]\cr
=&k\Bigl[-\frac{P(\bar{r},\bar{y}'')}{\lambda^3}\int_{\R^N}U_{0,1}^2+O\Bigl(\frac{1}{\lambda^{3+\epsilon}}\Bigr)\Bigr].
\end{align*}
Similarly, we have
\begin{align*}
\int_{\R^N}Q(y)W^*_{2,\lambda}\frac{\partial W^*_{2,\lambda}}{\partial\lambda}=k\Bigl[-\frac{Q(\bar{r},\bar{y}'')}{\lambda^3}\int_{\R^N}V_{0,1}^2+O\Bigl(\frac{1}{\lambda^{3+\epsilon}}\Bigr)\Bigr].
\end{align*}
On the other hand, we have
\begin{align*}
&\int_{\R^N}\Bigl({W^*}_{1,\lambda}^{2^*-1}-\sum\limits_{j=1}^k{W^*}_{1,x_j}^{2^*-1}\Bigr)\frac{\partial W^*_{1,\lambda}}{\partial\lambda}
=k\int_{\Omega_1}\Bigl({W^*}_{1,\lambda}^{2^*-1}-\sum\limits_{j=1}^k{W^*}_{1,x_j}\Bigr)\frac{\partial {W^*}_{1,\lambda}}{\partial\lambda}\cr
=&k\int_{\Omega_1}(2^*-1){W^*}_{1,x_1}^{2^*-2}\sum\limits_{j=2}^k{W^*}_{1,x_j}\frac{\partial {W^*}_{1,x_1}}{\partial \lambda}+O\Bigl(\frac{1}{\lambda^{3+\epsilon}}\Bigr)\cr
=&k\Bigl(-\sum\limits_{j=2}^k\frac{C_1}{\lambda^{(N-1)}|x_1-x_j|^{N-2}}\Bigr)+O\Bigl(\frac{1}{\lambda^{3+\epsilon}}\Bigr)
\end{align*}
for some constant $C_1=\ds\int_{\R^N}W_{0,1}^{2^*}>0$.

In view of $W^*_{2,\lambda}=kW^*_{1,\lambda}$, we estimate
\begin{align*}
 \int_{\R^N}\Bigl({W^*}_{1,\lambda}^{\frac{2^*}{2}-1}{W^*}_{2,\lambda}^{\frac{2^*}{2}}&-\sum\limits_{j=1}^k{W^*}_{1,x_j}^{\frac{2^*}{2}-1}{W^*}_{2,x_j}^{\frac{2^*}{2}}\Bigr)\frac{\partial {W^*}_{1,\lambda}}{\partial\lambda}
 =O\int_{\R^N}\Bigl({W^*}_{1,\lambda}^{2^*-1}-\sum\limits_{j=1}^k{W^*}_{1,x_j}^{2^*-1}\Bigr)\frac{\partial {W^*}_{1,\lambda}}{\partial\lambda}\cr
 =&kO\Bigl(-\sum\limits_{j=2}^k\frac{\kappa^{\frac{2^*}{2}}\ds\int_{\R^N}W_{0,1}^{2^*}}{\lambda^{N-1}|x_1-x_j|^{N-2}}\Bigr)+O\Bigl(\frac{1}{\lambda^{3+\epsilon}}\Bigr)
\end{align*}
and
\begin{align*}
 \int_{\R^N}\Bigl(W_{1,\lambda}^{\frac{2^*}{2}}W_{2,\lambda}^{\frac{2^*}{2}-1}-&\sum\limits_{j=1}^kW_{1,x_j}^{\frac{2^*}{2}}W_{2,x_j}^{\frac{2^*}{2}-1}\Bigr)\frac{\partial W_{2,\lambda}}{\partial\lambda}
 =O\int_{\R^N}\Bigl(W_{1,\lambda}^{2^*-1}-\sum\limits_{j=1}^kW_{1,x_j}^{2^*-1}\Bigr)\frac{\partial W_{2,\lambda}}{\partial\lambda}\cr
 =&kO\Bigl(-\sum\limits_{j=2}^k\frac{\kappa^{\frac{2^*}{2}\ds\int_{\R^N}W_{0,1}^{2^*}}}{\lambda^{N-1}|x_1-x_j|^{N-2}}\Bigr)+O\Bigl(\frac{1}{\lambda^{3+\epsilon}}\Bigr).
\end{align*}
So, we obtain
\begin{align*}
I_1=k\Bigl(-\frac{B_1}{\lambda^3}\bigl(P(\bar{r},\bar{y}'')+\kappa^2Q(\bar{r},\bar{y}'')\bigr)+\sum\limits_{j=2}^k\frac{C_1(1+\beta \kappa^{\frac{2^*}{2}})}{\lambda^{N-1}|x_1-x_j|^{N-2}}\Bigr)+O\Bigl(\frac{1}{\lambda^{2+\epsilon}}\Bigr),
\end{align*}
where $B_1=\ds\int_{\R^N}W_{0,1}^{2}$, $C_1=\ds\int_{\R^N}W_{0,1}^{2^*}$ and $1+\beta \kappa^{\frac{2^*}{2}}>0$.
\end{proof}

\begin{lemma}\label{lmA.2}
If $N\geq5$, then
\begin{align*}
\frac{\partial I(W_{1,\lambda},W_{2,\lambda})}{\partial\bar{r}}=k\Bigl(\frac{B_1}{\lambda^2}\frac{\partial (P(\bar{r},\bar{y}'')+\kappa^2Q(\bar{r},\bar{y}''))}{\partial\bar{r}}+\sum\limits_{j=2}^k\frac{C_1(1+\beta \kappa^{\frac{2^*}{2}})}{\lambda^{N-2}|x_1-x_j|^{N-2}}\Bigr)
\end{align*}
and
\begin{align*}
\frac{\partial I(W_{1,\lambda},W_{2,\lambda})}{\partial \bar{y}''}=k\Bigl(\frac{B_1}{\lambda^2}\frac{\partial (P(\bar{r},\bar{y}'')+\kappa^2Q(\bar{r},\bar{y}''))}{\partial \bar{y}''_k}\Bigr)
+O\Bigl(\frac{1}{\lambda^{1+\epsilon}}\Bigr),
\end{align*}
where $B_1$ and $C_1$ are the same positive constants in Lemma \ref{lmA.1}.
\end{lemma}
For each fixed $k$ and $j$, we consider the following function
\begin{align*}
g_{k,j}=\frac{1}{(1+|y-x_j|)^\alpha}\frac{1}{(1+|y-x_k|)^\beta},
\end{align*}
where $\alpha\geq1$ and $\beta\geq1$ are two constants.
\begin{lemma}\label{lmA.3}
\text{\rm (\cite{Wei-Yan}, Lemma B.1)} For any constants $0<\delta\leq\min\{\alpha,\beta\}$, there is a constant $C>0$, such that
\begin{align*}
g_{k,j}(y)\leq\frac{C}{|x_k-x_j|^\delta}\Bigl(\frac{1}{(1+|y-x_k|)^{\alpha+\beta-\delta}}+\frac{1}{(1+|y-x_j|)^{\alpha+\beta-\delta}}\Bigr).
\end{align*}
\end{lemma}
\begin{lemma}\label{lmA.4}
\text{\rm (\cite{Wei-Yan}, Lemma B.2)} For any constant $0<\delta<N-2$, there is a constant $C>0$, such that
\begin{align*}
\int_{\R^N}\frac{1}{|y-z|^{N-2}}\frac{1}{(1+|z|)^{2+\delta}}\leq \frac{C}{(1+|y|)^\delta}.
\end{align*}
\end{lemma}
\begin{lemma}\label{lmA.5}
\text{\rm (\cite{Wei-Yan}, Lemma B.3)} Suppose that $N\geq5$, then there is a small constant $\theta>0$, such that
\begin{align*}
\int_{\R^N}\frac{1}{|y-z|^{N-2}}W_{1,\lambda}^{\frac{4}{N-2}}\sum\limits_{j=1}^k\frac{1}{(1+\lambda|z-x_j|)^{\frac{N-2}{2}+\tau}}
\leq \sum\limits_{j=1}^k\frac{C}{1+\lambda|y-x_j|^{\frac{N-2}{2}+\tau+\theta}}.
\end{align*}
\end{lemma}
\begin{lemma}\label{lmA.6}
For any solution $(\varphi,\psi)$ to problem \eqref{eqnonlinear0} with $\|(\varphi,\psi)\|_{*}<+\infty$, there must hold further that
\begin{align*}
|\varphi|\leq\frac{1}{2}U,\,\,\,\,|\psi|\leq\frac{1}{2}V.
\end{align*}
\end{lemma}
\begin{proof}
Since the proof is similar as Lemma of \cite{Guo-Wang-Wang}, we omit it.
\end{proof}
Acknowledgement: The authors would like to thank Professor Chunhua Wang for the helpful  discussion with her. This paper is supported by NSFC(No. 12071169 and No.12471106).

\bibliography{reference}

\end{document}